\documentclass[a4paper,11pt]{article}

\usepackage{graphicx}
\usepackage{amsmath}
\usepackage{amsfonts}
\usepackage{titlesec}
\usepackage{amsthm}
\usepackage{amssymb}
\usepackage{geometry}
\usepackage{color}
\usepackage{accents}
\usepackage{epic}

\usepackage{algorithm2e}

\newcommand{\R}{\mathbb{R}}
\newcommand{\C}{\mathbb{C}}
\newcommand{\T}{\mathcal{T}}
\newcommand{\N}{\mathbb{N}}

\newcommand{\ci}{\mathrm{i}} 

\newcommand{\PR}{P_{h}}

\newcommand{\triangleL}{L}

\newcommand{\Ltwo}[2]{\langle #1 , #2 \rangle_{L^2({\mathcal{D}})}}
\newcommand{\energy}[1]{\|{#1}\|_{E({\mathcal{D}})}}
\newcommand{\deltat}[1]{\tau_{#1}}

\newcommand{\gmin}{\gamma_{\operatorname{min}}}
\newcommand{\gmax}{\gamma_{\operatorname{max}}}
\newcommand{\hmin}{h_{\mbox{{\rm\tiny min}}}}
\newcommand{\hmax}{h_{\mbox{{\rm\tiny max}}}}

\definecolor{dark-green}{rgb}{0.0,0.4,0.0}

\newtheorem{theorem}{Theorem}[section]

\newtheorem{lemma}[theorem]{Lemma}
\newtheorem{proposition}[theorem]{Proposition}

\theoremstyle{definition}
\newtheorem{definition}[theorem]{Definition}
\newtheorem{remark}[theorem]{Remark}
\newtheorem{conclusion}[theorem]{Conclusion} %
\usepackage{framed}

\newenvironment{fshaded}{%
\MakeFramed {\FrameRestore}}%
{\endMakeFramed}

\newtheorem{cstep}{Step}

\begin{document}

\begin{center}
{\LARGE The Finite Element Method for the time-dependent Gross-Pitaevskii equation with angular momentum rotation\renewcommand{\thefootnote}{\fnsymbol{footnote}}\setcounter{footnote}{0}
 \hspace{-3pt}\footnote{This work was supported by 
 the Swedish Research Council.}}\\[2em]
\end{center}

\renewcommand{\thefootnote}{\fnsymbol{footnote}}
\renewcommand{\thefootnote}{\arabic{footnote}}

\begin{center}
{\large Patrick Henning\footnote[1]{Department of Mathematics, KTH Royal Institute of Technology, 
SE-100 44 Stockholm, Sweden},
Axel M\r{a}lqvist\footnote[2]{Department of Mathematical Sciences, Chalmers University of Technology and University of Gothenburg, SE-412 96 G\"oteborg, Sweden}}\\[2em]
\end{center}

\begin{center}
{\large{\today}}
\end{center}

\begin{center}
\end{center}

\begin{abstract}
We consider the time-dependent Gross-Pitaevskii equation describing the dynamics of rotating Bose-Einstein condensates and its discretization with the finite element method. We analyze a mass conserving Crank-Nicolson-type discretization and prove corresponding a priori error estimates with respect to the maximum norm in time and the $L^2$- and energy-norm in space. The estimates show that we obtain optimal convergence rates under the assumption of additional regularity for the solution to the Gross-Pitaevskii equation. We demonstrate the performance of the method in numerical experiments.
\end{abstract}

\section{Introduction}

When a dilute gas of a certain type of Bosons is trapped by a potential and afterwards cooled down to extremely low temperatures close to the absolute minimum of $0$ Kelvin, a so called Bose-Einstein condensate (BEC) is formed \cite{Bos24,DGP99,Ein24,PiS03}. Such a condensate consists of particles
that occupy the same quantum state.
That means that they are no more distinguishable from each other and that they behave in their collective like one single 'super-atom'. Recent overviews on the mathematics for Bose-Einstein condensates are given in \cite{BaC13b,Bao14}.

In this work, we focus on the specific case of Bose-Einstein condensates in a rotational frame \cite{FSF01}. One of the interesting features of a Bose-Einstein condensate is its superfluid behavior. In order to distinguish a superfluid from a normal fluid at the quantum level, one needs to verify the formation of vortices with a quantized circulation (cf. \cite{Aft06} for an introduction in the context of BECs). In experimental setups the formation of such vortices may be triggered by rotating the condensate. This can be achieved by using a stirring potential which is generated by imposing laser beams on the magnetic trap (cf. \cite{AfM10,MCW00,MCB01,SBC04,MAH99,SBB05}). If the rotational speed is sufficiently large, the vortices can be detected (cf. \cite{ARV01}). In particular, the equilibrium velocity of the BEC
can no longer be identified with
a solid body rotation
and it can be observed that the rotational symmetry breaks (cf. \cite{Sei02} for an analytical proof). The number of vortices strongly depends on the rotation frequency. However, if the rotational speed is too low no vortices arise and if the rotational speed is too large (relative to the strength of the trapping potential) the BEC can be destroyed by centrifugal forces. Analytical results concerning the formation, or lack, of vortices, their stability, types and structures depending on the rotational speeds and trapping potentials is found in \cite{AJR11,BWM05,CoR13,LiS06,Rou11,Sei02}. Detailed numerical investigations are given in \cite{BDZ06,BWM05}.

The formation and the dynamics of BECs are typically modeled by the Gross-Pitaevskii equation (GPE) which is a Schr\"odinger equation with an additional nonlinear term that accounts for particle-particle interactions \cite{Gro61,LSY01,Pit61}. To account for a rotating BEC, it is common to extend this model by an angular momentum term.
Let ${\mathcal{D}} \subset \R^d$, $d=2,3$, be a bounded convex Lipschitz domain and $[0,T] \subset \R$ a time interval. We consider the dimensionless time-dependent Gross-Pitaevskii equation. For the case $d=3$ we seek the complex-valued wave function $u : {\mathcal{D}} \times [0,T] \rightarrow \C$ that describes the quantum state of the condensate. It is the solution with initial state $u(\cdot,0)=u_0$ to the nonlinear Schr\"odinger equation
\begin{align}
\label{gp-equation-with-rotation-full}
\ci \partial_t u &= -\frac{1}{2} \triangle u + V \hspace{2pt} u + \ci \boldsymbol \Omega \cdot \left( \mathbf{x} \times \nabla \right) u + \beta |u|^2 u \qquad \mbox{in} \enspace {\mathcal{D}},\\
\nonumber u&=0 \qquad \hspace{182pt} \mbox{on} \enspace \partial {\mathcal{D}},
\end{align}
where we denote $\mathbf{x}=(x,y,z)\in\R^3$. Here, $V$ characterizes the magnetic trapping potential that confines the system (by adjusting $V$ to some trap frequencies) and the nonlinear term $\beta |u|^2 u$ describes the species of the bosons and how they interact with each other. In particular, $\beta$ depends on the number of bosons, their individual mass and their scattering length. We assume that $\beta$ is strictly positive (which means that we have a repulsive interaction between the particles). The term $\ci \boldsymbol \Omega \cdot \left( \mathbf{x} \times \nabla \right) u$ characterizes the angular rotation of the condensate, where $\boldsymbol{\Omega}\in \R^3$ defines the angular velocity. As usual, the operator $\mathbf{L}=(\mathcal{L}_x,\mathcal{L}_y,\mathcal{L}_z):=- \ci \left( \mathbf{x} \times \nabla \right)=\mathbf{x} \times \mathbf{P}$  describes the angular momentum, with $\mathbf{P}=- \ci \nabla$ denoting the momentum operator.

In the following, we assume that the rotation is around the $z$-axis, which leads to the simplification
$\ci \boldsymbol \Omega \cdot \left( \mathbf{x} \times \nabla \right) = - \Omega \mathcal{L}_z$, where $\mathcal{L}_z = - \ci \left( x \partial_y - y \partial_x \right)$ is the $z$-component of the angular momentum. With this simplification the weak formulation of problem (\ref{gp-equation-with-rotation-full}) (respectively its dimension reduced version in $2d$) reads: find $u \in C^0([0,T),H^1_0({\mathcal{D}}))$ and $\partial_t  u \in C^0([0,T),H^{-1}({\mathcal{D}}))$ such that $u(\cdot,0)=u_0$ and
\begin{eqnarray}
\label{weak-formulation-gpe}\lefteqn{\ci \Ltwo{\partial_t u(\cdot,t)}{\phi} = \frac{1}{2} \Ltwo{\nabla u(\cdot,t)}{\nabla \phi}}\\
\nonumber&\enspace& + \Ltwo{V u(\cdot,t)}{ \phi}
- \Omega \Ltwo{\mathcal{L}_z u(\cdot,t) }{\phi}
+\beta  \Ltwo{|u(\cdot,t)|^2 u(\cdot,t)}{\phi}
\end{eqnarray}
for all $\phi \in H^1_0({\mathcal{D}})$ and almost every $t \in (0,T)$. Here, $\Ltwo{ \cdot }{ \cdot }$ denotes the standard $L^2$-scalar product for complex valued functions, i.e. $\Ltwo{ v }{ w } = \int_{\mathcal{D}} v(\mathbf{x}) \overline{w(\mathbf{x})} \hspace{2pt} d \mathbf{x}$ for $v,w \in L^2({\mathcal{D}})$.

A recent existence and uniqueness result concerning the solution of (\ref{weak-formulation-gpe}) can be found in \cite{AMS12} for the case of the three dimensional Cauchy problem, i.e. for the case ${\mathcal{D}}=\R^3$ (see also \cite{HHL07} for an earlier work). A general comprehensive overview on existence and uniqueness of nonlinear Schr\"odinger equations can be found in the book by Cazenave \cite{Caz03}.

The literature on the numerical treatment of (\ref{weak-formulation-gpe}) is rather limited for the case $\Omega\neq 0$. Very efficient methods that exploit Fourier expansions were proposed in \cite{BaW06,BLS09,BMT13}: in \cite{BLS09} a time-splitting method is proposed that is based on the scaled generalized-Laguerre, Fourier and Hermite functions, whereas in \cite{BMT13} it is suggested to discretize (\ref{weak-formulation-gpe}) in rotating Lagrangian coordinates. A finite difference discretization is discussed in \cite{BaC13}. A comparative overview on different time-discretization is given in \cite{ABB13c}. Concerning the numerical treatment of the eigenvalue problem associated with (\ref{weak-formulation-gpe}), we refer to \cite{AnD14,DaK10}.

Even though spectral and pseudo-spectral methods (such as the explicit methods proposed in \cite{BaW06,BLS09,BMT13}) are typically computationally cheaper than a pure finite element based approach as proposed in this paper, they generally require a high smoothness of the magnetic potential to work. Non-smooth potentials can for instance arise in the context of investigating Josephson effects (cf. \cite{WWW98,ZSL98}) or experiments involving very rough disorder potentials (cf. \cite{NBP13}). In corresponding numerical simulations, the usage of finite elements seems to be unavoidable for an efficient method. Another advantage of finite elements is that they can be easily combined with mesh adaptivity, as it is helpful to resolve localized vortices.
%


There are some results concerning the convergence of numerical methods with respect to the space discretization. Concerning $P1$ finite elements and for the particular case that $V=0$, $\Omega=0$ and that the spatial mesh is quasi-uniform, a priori error estimates can be found in \cite{ADK91,KaM98,KaM99,San84,Tou91,Wan14,Zou01}. We describe those results in chronological order.

The first results were obtained by Sanz-Serna \cite{San84} who considered a modified Crank-Nicolson scheme that conserves the mass and the energy. For the case $d=1$ and a periodic boundary condition, optimal $L^2$-error estimates were derived with a quadratic order convergence in time. A necessary condition for the analysis was that the time step size $\deltat{}$ can be bounded by the mesh size $h$, i.e. the time step size is constrained by $\deltat{}=\mathcal{O}(h)$.

 In \cite{ADK91}, Akrivis et al. generalize the a priori $L^2$-error estimates of \cite{San84} to $d=1,2,3$ and to the case of a homogenous Dirichlet boundary condition. Furthermore they could relax the constraint for the time step size to the condition $\deltat{}=\mathcal{O}(h^{d/4})$. Beside the modified Crank-Nicolson scheme, the authors also study a one-stage Gauss-Legendre implicit Runge-Kutta scheme (IRK) that we will also consider in this paper. The IRK is still mass-conservative, but does no longer conserve the energy. 
However, as we will see numerically, the energy deviation is marginal. Again, the condition $\deltat{}=\mathcal{O}(h^{d/4})$ is required. Furthermore, the authors propose and analyze a Newton-scheme for solving the nonlinear problems that arise in each time step.

In \cite{Tou91}, Tourigny investigates the case of optimal $L^{\infty}$- and $H^1$-error estimates. He analyzes the same implicit Runge-Kutta scheme as considered in \cite{ADK91} and recovers the constraint $\deltat{}=\mathcal{O}(h^{d/4})$. Furthermore, he investigates a classical Backward-Euler discretization for which the more severe constraint $\deltat{}=\mathcal{O}(h^{d/2})$ is required. However, as we will see in our analysis below, both constraints are not optimal. For instance for the Backward-Euler scheme, for any $s>1$, we can improve it to $\deltat{}=\mathcal{O}(|\ln{h}|^{-s/2})$ for $d=2$ and to $\deltat{}=\mathcal{O}(h^{s/2})$ for $d=3$ (see Theorem \ref{final-a-priori-error-estimate} below).

Concerning higher order schemes (without conservation properties), a space-time finite element method
was proposed and analyzed in \cite{KaM98,KaM99} for the case $d=2$ and for
graded meshes.
Here, \cite{KaM98} is devoted to the case of a Discontinuous Galerkin time discretization and \cite{KaM99} is devoted to a Continuous Galerkin time discretization. Optimal error estimates in $L^2$ and $H^1$ are derived. Here the constraint (for $d=2$) reads $\deltat{}^p=\mathcal{O}(|\ln{h}|^{-s})$ for some $s>1$ and where $p$ denotes the polynomial degree used for the time discretization. Hence, it excludes lowest order schemes such as the Backward-Euler scheme for which $p=0$.

In \cite{Zou01}, Zouraris considers a mass conservative linearly implicit two-step finite element method. Zouraris proves optimal order $L^2$- and $H^1$-error estimates under the mild time step conditions $\deltat{}=\mathcal{O}(|\ln{h}|^{-1/3})$ for $d=2$ and $\deltat{}=\mathcal{O}(h^{1/3})$ for $d=3$.

In a recent work \cite{Wan14}, Wang studies a new type of a linearized Crank-Nicolson discretization which is mass but not energy conservative. Again, optimal order $L^2$-error estimates are derived however with the breakthrough that no constraints for the coupling between time step size and mesh size are required. The condition of quasi-uniformity is still necessary.

Concerning the convergence of space discretizations for the nonlinear GPE eigenvalue problem (again for $\Omega=0$)
we refer to \cite{CCM10} for optimal convergence rates in Fourier and finite elements spaces and to \cite{HMP14b} for a two level discretization technique based on suitable orthogonal decompositions.
Regarding the Gross-Pitaevskii equation with rotation term (i.e. $\Omega \neq 0$), we are only aware of the work by Bao and Cai \cite{BaC13} where optimal error estimates for the finite difference method are proved. So far, there seem to be no results concerning finite element approximations.

In this work we present an error analysis for a Crank-Nicolson-type finite element approximation of the time-dependent GPE with rotation. More precisely, we analyze the one-stage Gauss-Legendre implicit Runge-Kutta scheme earlier considered by Akrivis et al. \cite{ADK91} and Tourigny \cite{Tou91}. We generalize these work with respect to two points: we consider the equation with potential and with an angular momentum rotation,
and for arbitrary $s>1$ we show that the time step constrained $\deltat{}=\mathcal{O}(h^{d/4})$ can be relaxed to $\deltat{}=\mathcal{O}(|\ln{h}|^{-s/4})$ for $d=2$ and to $\deltat{}=\mathcal{O}(h^{s/4})$ for $d=3$.
We do not consider Fourier approaches here (even though they can be computationally more efficient in many applications),
since they require smoothness of the trapping potential, whereas the strength of finite element approaches lies in the fact that they do not require such smoothness and that it can be easily combined with adaptive mesh refinement strategies. This might be necessary in experiments involving disorder potentials.

{\bf Outline.} In Section \ref{section-mod-prob-and-prelim} we establish our model problem and state the basic preliminaries. The main results are presented in Section \ref{section-disc-and-main-result}, where we state a Crank-Nicolson-type time- and P1 Finite Element space discretization of the Gross-Pitaevskii equation. Furthermore, corresponding a priori error estimates in the $L^{\infty}(L^2)$-norm and in the $L^{\infty}(H^1)$-norm are given. The proof of
these estimates
takes place in several steps.
First we introduce a general framework and some auxiliary results by investigating the fully continuous problem in weak formulation.
This is done in Section \ref{section-aux-results-analytical}.
In Section \ref{section-exst-uni-disc-sol} we show well-posedness of the numerical scheme presented in Section \ref{section-disc-and-main-result}.
Furthermore, we introduce a regularized discrete auxiliary problem which will turn out to produce the same solutions as the
considered
Crank-Nicolson-type
Finite Elemente scheme
(under suitable assumptions).
Finally, in Section \ref{final-apriori-proofs} we derive an error identity and estimate the arising terms. At the end of this section, all results are combined to finish the proof of the main theorem. 
We conclude the paper with numerical experiments in Section \ref{section-num-sec}.

\section{Model problem and preliminaries}
\label{section-mod-prob-and-prelim}

Let $d=2,3$ denote the space dimension. In order to keep our analysis as general as possible, we subsequently consider a slightly generalized Gross-Pitaevskii model. Before stating the problem and a corresponding set of assumptions, we introduce our basic notation.

By $\overline{x}$ we denote the complex conjugate of a complex number $x \in \C$, by $\mathbf{x} \cdot\mathbf{y}$ we denote the Euclidean scalar product between $\mathbf{x},\mathbf{y}\in \C^d$ (i.e. $ \mathbf{x} \cdot \mathbf{y}:=\sum_{i=1}^d x_i \overline{y}_i$) and by $|\mathbf{x}|:=\sqrt{\mathbf{x} \cdot\mathbf{x}}$ we denote the corresponding norm. The real part of a complex number is denoted by $\Re$ and by $\Im$ its imaginary part. We furthermore use the standard notation for the Sobolev spaces $W^{k,p}({\mathcal{D}})$ (for $0\le k < \infty$ and $1 \le p \le \infty$) equipped with the norm
\begin{align*}
\| v \|_{W^{k,p}({\mathcal{D}})}:=
\begin{cases}
\left( \sum_{|\boldsymbol{\alpha}|\le k} \| \partial^{\boldsymbol{\alpha}} v \|_{L^p({\mathcal{D}})}^p \right)^{1/p}, \quad &\mbox{for } 1\le p < \infty,\\
\max_{|\boldsymbol{\alpha}|\le k} \| \partial^{\boldsymbol{\alpha}} v \|_{L^{\infty}({\mathcal{D}})}, \quad &\mbox{for } p = \infty.
\end{cases}
\end{align*}
For $p=2$ we write as usual $H^k({\mathcal{D}}):=W^{k,2}({\mathcal{D}})$. The semi-norms on $H^{k}({\mathcal{D}})$ are denoted by
$$|v|_{H^k({\mathcal{D}})} := \left( \sum_{|\boldsymbol{\alpha}|= k} \| \partial^{\boldsymbol{\alpha}} v \|_{L^2({\mathcal{D}})}^2 \right)^{1/2}.$$

We consider the following model problem.
\begin{definition}[Model problem]
We consider the (smooth) linear differential operator $L:H^1_0({\mathcal{D}})\rightarrow H^{-1}({\mathcal{D}})$ that is associated with the following bilinear form,
\begin{align}
\label{smooth-linear-operator}\langle L(v),w \rangle_{H^{-1}({\mathcal{D}}),H^1({\mathcal{D}})} := \int_{\mathcal{D}} A(\mathbf{x}) \nabla v(\mathbf{x}) \cdot \overline{\nabla w(\mathbf{x})} + \ci b(\mathbf{x}) \cdot \nabla v(\mathbf{x}) \overline{w(\mathbf{x})} + c(\mathbf{x}) v(\mathbf{x}) \overline{w(\mathbf{x})} \hspace{2pt} d\mathbf{x}.
\end{align}
With the above definition we seek
$u \in L^{\infty}([0,T),H^1_0({\mathcal{D}}))$ and $\partial_t  u \in L^{\infty}([0,T),H^{-1}({\mathcal{D}}))$
such that $u(\cdot,0)=u_0$ and
\begin{eqnarray}
\label{model-problem}\ci \Ltwo{\partial_t u(\cdot,t)}{w} =
\langle L(u(\cdot,t)),w \rangle_{H^{-1}({\mathcal{D}}),H^1({\mathcal{D}})}
+ \Ltwo{(\kappa(\cdot) + \beta |u(\cdot,t)|^2) u(\cdot,t)}{w}
\end{eqnarray}
for all $w \in H^1_0({\mathcal{D}})$ and almost every $t \in (0,T)$. Note that any such solution automatically fulfills $u \in C^0([0,T),L^2({\mathcal{D}}))$ so that $u(\cdot,0)=u_0$ makes sense.
\end{definition}

Here we make the following assumptions.
\begin{itemize}
\item[(A1)] The computational domain ${\mathcal{D}} \subset \R^d$ (for $d=2,3$) is a convex bounded polyhedron.
\item[(A2)] The coefficients $A$, $b$ and $c$ are real valued, smooth and bounded (i.e. $L$ represents the smooth linear part of the problem). On the other hand, we assume $\kappa \in L^{\infty}({\mathcal{D}},\C)\cap W^{1,3}({\mathcal{D}},\C)$; $\beta \in \R_{\ge 0}$ and $u_0 \in H^2({\mathcal{D}})\cap H^1_0({\mathcal{D}})$.
\item[(A3)] The real matrix-valued coefficient $A=A(\mathbf{x})$ is symmetric and there exist positive constants
$\gmin>0$ and $\gmax\geq\gmin$ such that
\begin{equation}\label{e:spectralbound}
\gmin |\xi|^2 \le A(\mathbf{x}) \xi \cdot \xi \le \gmax |\xi|^2 \quad\text{for all }(\xi,\mathbf{x})\in \R^d \times {\mathcal{D}}.
\end{equation}
By the properties of $A$
there exists a pointwise invertible matrix-valued coefficient $A^{1/2}$ such that $A^{1/2} A^{1/2} = A$. We denote its inverse by $A^{-1/2}:=(A^{1/2})^{-1}$.
\item[(A4)] The real vector-valued coefficient $b=b(\mathbf{x})$ is divergence free, i.e. $\nabla \cdot b=0$.
\item[(A5)] It holds $\Re(\kappa)\ge 0$ and the real coefficient $c=c(\mathbf{x})$ is such that there exist real-valued constants $\zeta_0>0$ and $\zeta_1>1$ with
$$4 c(\mathbf{x}) - (2+\zeta_1) |A^{-1/2}(\mathbf{x}) b(\mathbf{x})|^2 \ge 4 \zeta_0 >0
\qquad \mbox{for all} \enspace \mathbf{x} \in \mathcal{D}.$$
\end{itemize}
We note that assumptions (A1)-(A4) are trivially fulfilled for the Gross-Pitevskii equation (\ref{weak-formulation-gpe}). In practice, $A$ is typically just a constant, whereas $b$ describes the angular momentum rotation like in \eqref{weak-formulation-gpe}. The term $c$ describes any kind of real-valued non-negative smooth potential such as harmonic potentials of the structure $c(\mathbf{x})=\gamma_x^2 x^2 + \gamma_y^2 y^2 + \gamma_z^2 z^2$ with scaled trapping frequencies $\gamma_x,\gamma_y,\gamma_z\in \R$. The coefficient $\kappa$ can be used to model arrays of quantum wells for investigating Josephson oscillations (see \cite{WWW98,ZSL98}) or any other type of rough potential. Furthermore, $\kappa$ can be also used to describe imaginary potentials, see for instance the complex double-well potential in \cite{FDH14} or applications with phenomenological damping.

Observe that (A4) implies that the operator $L$ is self-adjoint.
Assumption (A5) is an additional (often crucial) physical constraint, which says that
the rotational speed $\Omega$ should be balanced by the trapping potential $V$ in the sense that
$V - 3/2 |\Omega|^2 \left( x^2 + y^2 \right) \ge \zeta_0 > 0$ on ${\mathcal{D}}$. The physical interpretation is that the trapping potential should be stronger than the arising centrifugal forces. Otherwise particles can escape from the trap and the Bose-Einstein condensate is destroyed (hence there exist no physically meaningful solutions). As we will see later, the differential operator $L$ is elliptic, but degenerates for the case $\zeta_0=0$ and $\zeta_1=1$, which just resembles the instability.

\begin{remark}
Observe that assumption (A5) allows to balance $c$ and $\kappa$ in a suitable way. For instance, if we only have $c \ge 0$ but $4 \Re(\kappa) - (2+\zeta_1) |A^{-1/2} b|^2 \ge 4 \zeta_0 >0$, we can define $\kappa^{\mbox{\tiny{\rm new}}}(\mathbf{x}):=\kappa(\mathbf{x}) - 4^{-1} (2+\zeta_1) |A^{-1/2}(\mathbf{x}) b(\mathbf{x})|^2 - \zeta_0$ and accordingly
$c^{\mbox{\tiny{\rm new}}}(\mathbf{x}):=c(\mathbf{x}) + \zeta_0 + 4^{-1} (2+\zeta_1) |A^{-1/2}(\mathbf{x}) b(\mathbf{x})|^2$, which again suit our assumptions above. Also note that we can hide any imaginary part of $c$ in $\kappa$ (which is allowed to be imaginary without constraints).
\end{remark}

\begin{remark}[Existence and uniqueness]
In the case $A(\mathbf{x})=1$, $b(\mathbf{x})=0$, $c(\mathbf{x})=0$, $\kappa \in L^{\infty}(\mathcal{D},\R)$ and $u_0 \in H^1_0(\mathcal{D})$, equation \eqref{model-problem} admits at least one solution for any time $T>0$. The corresponding results can be e.g. found in \cite[Theorem 3.4.1, Corollary 3.4.2]{Caz03}.
If $d=2$ the solution is also unique (cf. \cite[Corollary 4.3.3 and Remark 3.6.4]{Caz03}).
Even though we are not aware of an explicit result that guarantees existence of a solution to problem \eqref{model-problem} under the more general assumptions (A1)-(A5), it 
appears 
straightforward 
by exploiting Galerkin's method and compactness results via energy conservation (cf. \cite{Caz03} or \cite[Chapter 7.1 and 7.2]{Eva10}).
\end{remark}

\section{Discretization and main result}
\label{section-disc-and-main-result}

In this section we propose a space-time discretization of problem (\ref{model-problem}) and we state corresponding a priori error estimates in $L^{\infty}(L^2)$ and $L^{\infty}(H^1)$.

\subsection{Space discretization}
\label{subsection-space-disc}

In the following, we denote by $\T_h$ a conforming family of partitions of ${\mathcal{D}}\subset \R^d$ that consists of simplicial elements and which are shape regular, i.e. there exists an $h$-independent shape regularity parameter $\rho>0$ such that (for all $\T_h$) it holds
\begin{align}
\label{shape-regularity}
\mbox{\rm diam}(B_K) \ge \rho \hspace{2pt} \mbox{\rm diam}(K)
\end{align}
for all $K \in \T_h$, where $B_K$ denotes the largest ball contained in $K$. The diameter of an element $K \in \T_h$ is denoted by $h_K$; the maximum diameter by $\hmax:=\max_{K\in\T_H} h_K$ and the minimum diameter by $\hmin:=\min_{K\in\T_H} h_K$. Finally, by $h : \mathcal{D} \rightarrow \R_{>0}$ we denote the corresponding mesh function with $h(x):=h_K$ if $x \in K$. For brevity, we subsequently write $\| h v \|_{H^k({\mathcal{D}})}$ for some $v \in H^k({\mathcal{D}})$ to abbreviate $\left( \sum_{K \in \T_h} h_K^2 \| v \|_{H^k(K)}^2 \right)^{1/2}$.
The considered P1 Lagrange finite element space $S_h \subset H^1_0({\mathcal{D}})$ is given by
\begin{align}
S_h := \{v \in H^1_0({\mathcal{D}}) \;\vert \;\forall K\in\T_h,v\vert_K \text{ is a complex-valued polynomial of total degree}\leq 1\}.
\end{align}
By $\{ \lambda_1, \ldots, \lambda_{N_h} \}$ we denote an ordered (Lagrange) basis of $S_h$.
In particular, we denote by $N_h=$dim$(S_h)$ the number of degrees of freedom in $S_h$
(which is twice the number of interior nodes in $\T_h$).
On $S_h$, we introduce the corresponding $L^2$-projection and the Ritz-projection associated with $L$.

\begin{definition}[$L^2$-projection]
\label{def-ritz-projection}The $L^2$-projection $P_{L^2} : H^1_0({\mathcal{D}}) \rightarrow S_h$ is given by
\begin{align*}
\mbox{for } v \in H^1_0({\mathcal{D}}): \hspace{25pt} \Ltwo{ P_{L^2}(v) }{  w_h } = \Ltwo{ v }{  w_h }  \qquad \mbox{for all } w_h \in S_h.
\end{align*}
\end{definition}

\begin{definition}[Ritz projection]
For $v \in H^1_0({\mathcal{D}})$ the Ritz-projection $P_h(v) \in S_h$ associated with $L$ is given as the unique solution to the problem
\begin{align}
\label{definition-ritz-projection}\langle L(v - P_h(v)),w_h \rangle_{H^{-1}({\mathcal{D}}),H^1({\mathcal{D}})} =0 \qquad \mbox{for all } w_h\in S_h.
\end{align}
Existence and uniqueness of $P_h(v)$ follow from Conclusion \ref{ellipticity-of-L} below.
\end{definition}

In order to derive the final a priori error estimates, we require further assumptions on the grid $\T_h$, which will be posed indirectly in the following way exploiting the projections.

\begin{enumerate}
\item[(A6)] We assume that the $L^2$-projection is $H^1$-stable, i.e. there exists a
$h$-independent constant $C_{L^2}$
such that
\begin{align}
\label{H1-stability-L2-projection}
\| P_{L^2}(v) \|_{H^1({\mathcal{D}})} \le C_{L^2} \| v \|_{H^1({\mathcal{D}})} \qquad \mbox{for all } v\in H^1_0({\mathcal{D}}).
\end{align}
\item[(A7)] For $\mu>d$, we assume that the Ritz projection given by (\ref{definition-ritz-projection}) is $W^{1,\infty}$-stable for functions in $W^{2,\mu}({\mathcal{D}})$, i.e. there exists a $h$-independent constant $C_{W^{1,\infty}}$ such that
\begin{align}
\label{W-1-infty-stability}\| \nabla P_h(w) \|_{L^{\infty}({\mathcal{D}})} \le C_{W^{1,\infty}} \| \nabla w \|_{L^{\infty}({\mathcal{D}})}
\end{align}
for all $w \in H^1_0({\mathcal{D}}) \cap W^{2,\mu}({\mathcal{D}})$. Note that since ${\mathcal{D}}$ is a convex domain, we have the embedding $W^{2,\mu}({\mathcal{D}}) \hookrightarrow W^{1,\infty}({\mathcal{D}})$.
\end{enumerate}

Both assumptions (A6) and (A7) can be fulfilled by making suitable assumptions on $\T_h$. In this paper we directly assume
stability of the projections to avoid complicated mesh assumptions. Concerning (A6), recent results on the $H^1$-stability of $P_{L^2}$ on adaptively refined grids can be found in \cite{BaY14,KPP13,GHS14}. Concerning (A7), we refer to \cite[Theorem 8.1.11]{BrS08} where the result is established for quasi-uniform meshes. For results on graded meshes we refer to \cite{GLR09,DLS12}. We note that, the results on graded (locally quasi-uniform) meshes are only proved for the Laplacian operator, i.e. $L=-\triangle$, and its generalization to general elliptic operators is still open. However, it seems to be crucial that the operator $L$ is sufficiently smooth for (A7) to hold on graded meshes, this is why it might be important that $\kappa$ is not included in $L$ (in this context, see also the H\"older-estimates for the Green's functions proved in \cite{GLR09} and the necessary regularity assumptions made in \cite{MaR91}).

\subsection{Time discretization, method and main result}

In this paper we assume that the time interval $[0,T]$ is divided into $0=:t_0<t_1< \cdots < t_N := T$. Accordingly we define the $n$'th time interval by $I_n := (t_{n-1},t_{n}]$, the $n$'th time step by $\deltat{n} := t_{n} - t_{n-1}$ and step size function $\deltat{}\in L^{\infty}(0,T)$ by $\tau_{\vert I_n}:=\deltat{n}$. For simplicity we subsequently only write $\langle \cdot , \cdot \rangle:=\langle \cdot , \cdot \rangle_{H^{-1}({\mathcal{D}}),H^1({\mathcal{D}})}$ for the dual pairing on $H^1({\mathcal{D}})$. We consider the following one-stage Gauss-Legendre implicit Runge-Kutta scheme (which is of Crank-Nicolson-type). The scheme is mass conservative provided that $\Im(\kappa)=0$.

\begin{definition}[IRK Method for GPE]
\label{crank-nic-gpe}
Let $u_{h}^{0}:=\mathcal{I}_h(u_0) \in S_h$ be the Lagrange interpolation of $u_0$. For $n \ge 1$, we seek the approximation $u_{h}^{n} \in S_h$ with
\begin{eqnarray}
\label{cnd-problem}\Ltwo{ u_{h}^{n}}{v_h} +
\deltat{n} \hspace{2pt}\ci \hspace{2pt}
\langle L(u_h^{n-\frac{1}{2}}), v_h \rangle
 + \deltat{n} \hspace{2pt}\ci \hspace{2pt} \Ltwo{(\kappa + \beta |u_h^{n-\frac{1}{2}}|^2) u_h^{n-\frac{1}{2}}}{v_h}
= \Ltwo{ u_{h}^{n-1}}{v_h}
\end{eqnarray}
for all $v_h \in S_h$ and where $u_h^{n-\frac{1}{2}}:=(u_{h}^{n}+u_{h}^{n-1})/2$.
\end{definition}

For alternative time-discretizations based on operator splitting for nonlinear Schr\"odinger equations with a cubic nonlinearity we refer to \cite{Lub08,Gau11}, for the case without rotation, and to \cite{ABB13c}, for the case with rotation. More general approaches are discussed in \cite{HLW03}.

We note that the IRK scheme given by \eqref{cnd-problem} is mass conservative if $\Im(\kappa)=0$. The mass conservation, i.e. $\| u_{h}^{n} \|_{L^2(\mathcal{D})}=\| u_{h}^{0} \|_{L^2(\mathcal{D})}$ for all $n\ge 0$, is immediately seen by testing with $u_h^{n-\frac{1}{2}}$ in \eqref{cnd-problem} and taking the real part. Note that the conservation property implies that the scheme is unconditionally $L^2$-stable.

The following proposition follows from Lemma \ref{lemma-exist-uniq-disc-sols} and the proof of Theorem \ref{final-a-priori-error-estimate} below.
\begin{proposition}[Existence and uniqueness]
\label{prop-exist-and-unique}
If (A1)-(A5) are fulfilled and if $h$ and $\deltat{n}$ are small enough and such that $\ell_h (\hmax+\deltat{n}^2) \rightarrow 0$ for $h,\deltat{n}\rightarrow 0$, then there exists a solution $u_h^n$ of \eqref{cnd-problem}.
If $\Im(\kappa)=0$ and if $\deltat{n}$ is sufficiently small compared to $h$ (in the sense of Lemma \ref{lemma-exist-uniq-disc-sols} below), then the solution is also unique.
\end{proposition}

The main result of the work is the following a priori error estimate, which we prove in Section \ref{final-apriori-proofs}. Recall assumptions (A1)-(A5) from Section \ref{section-mod-prob-and-prelim} and (A6)-(A7) from Section \ref{subsection-space-disc}.

\begin{theorem}[Error estimates for the IRK discretization]
\label{final-a-priori-error-estimate}
Let assumptions (A1)-(A7) be fulfilled, let $u\in W^{2,\infty}(0,T;H^3({\mathcal{D}}))$
denote a solution of (\ref{model-problem}), and let $h$ and $\deltat{n}$ be such that $\ell_h (\hmax+\deltat{n}^2) \rightarrow 0$ for $h,\deltat{n}\rightarrow 0$ and where
\begin{align*}
\ell_h :=
\begin{cases}
 |\ln{\hmin}|^{1/2} \quad &\mbox{for } d=2 \\
 \hspace{3pt}|{\hmin}|^{-1/2} \hspace{3pt} \quad &\mbox{for } d=3.
\end{cases}
\end{align*}
Then, if $h$ and $\deltat{n}$ are small enough, there exist generic constants $C=C(u)$ that are independent of $h$, $\deltat{n}$ and $T$ such that for a solution $u_h^N$ of \eqref{crank-nic-gpe}
and for $m\in \{0,1\}$ it holds
\begin{eqnarray*}
\lefteqn{\| u(\cdot,T) - u_h^N \|_{H^m({\mathcal{D}})}
\le C  | h^{2-m} u(\cdot,T) |_{H^2({\mathcal{D}})} + C e^{CT} \left( | h^{2-m} u_0 |_{H^2({\mathcal{D}})}
+  \| h^{2-m} \partial_t u \|_{L^2(0,T,H^2({\mathcal{D}}))} \right)} \\
&\enspace& +
C e^{T}  \left( \sum_{k=1}^n
  \deltat{k}
 \left( \| h^{2-m} u \|_{L^{\infty}{(I_k,H^2({\mathcal{D}}))}}^2 
 +
 \| \deltat{}^2 \partial_{tt} u \|_{L^{\infty}(I_k,H^{2+m}({\mathcal{D}}))}^2
 + \| \deltat{}^2 u  \|_{W^{2,\infty}(I_k,H^{m}(\mathcal{D}))}^2 
 \right)
 \right)^{1/2}.
\hspace{200pt}
\end{eqnarray*}
\end{theorem}
We observe that the method yields optimal convergence rates, i.e. it is of quadratic order in space and time for the $L^2$-error and of linear order in space for the $H^1$-error. Details on the arising constants in Theorem \ref{final-a-priori-error-estimate} can be found in Lemma \ref{L2-estimate-full-E-h-n} and Lemma \ref{energy-estimate-full-E-h-n} below.

\begin{remark}
It is surprising that the $L^{\infty}(H^1)$-estimate in Theorem \ref{final-a-priori-error-estimate} requires the higher regularity $\partial_{tt} u(\cdot,t) \in H^3({\mathcal{D}})$. A similar observation has already been made by Karakashian and Makridakis \cite[Remark 4.3]{KaM98} for the simpler equation $\ci \partial_t u = - \triangle u + \beta |u|^2 u$. It should be investigated in the future if it is possible to weaken this regularity assumption.
\end{remark}

Finally, let us state the corresponding result that can be derived for the Backward-Euler Method. This result is rather for comparison, since the Backward-Euler is practically not desirable since it lacks both mass and energy conservation.

\begin{theorem}[Error estimates for a Backward-Euler discretization]
Assume (A1)-(A7), $u\in W^{1,2}(0,T;H^3({\mathcal{D}}))$ and $h$ and $\deltat{n}$ such that $\ell_h (\hmax+\deltat{n}) \rightarrow 0$ for $h,\deltat{n}\rightarrow 0$. Let further $u_{h}^{0}:=\mathcal{I}_h(u_0) \in S_h$. Then, for all small enough $h$ and $\deltat{n}$, there exists $u_{h}^{n} \in S_h$ with
\begin{eqnarray}
\label{bwe-problem}\Ltwo{ u_{h}^{n}}{v_h} +
\deltat{n} \hspace{2pt}\ci \hspace{2pt}
\langle L(u_h^n), v_h \rangle
 + \deltat{n} \hspace{2pt}\ci \hspace{2pt} \Ltwo{(\kappa + \beta |u_h^{n}|^2) u_h^{n}}{v_h}
= \Ltwo{ u_{h}^{n-1}}{v_h}
\end{eqnarray}
for all $v_h \in S_h$ and 
there exist generic constants $C=C(u)$ that are independent of $h$, $\deltat{n}$ and $T$ 
such that for $m\in \{0,1\}$ 
\begin{eqnarray*}
\lefteqn{\| u(\cdot,T) - u_h^N \|_{H^m({\mathcal{D}})}
\le C  | h^{2-m} u(\cdot,T) |_{H^2({\mathcal{D}})} + C e^{CT} \left( | h^{2-m} u_0 |_{H^2({\mathcal{D}})}
+ \| h^{2-m} \partial_t u \|_{L^2(0,T,H^2({\mathcal{D}}))} \right)}\\
&\enspace& +
C e^{C T} \left(  
\| h^{2-m} u \|_{L^2(0,T,H^2({\mathcal{D}}))} + \| \deltat{} \partial_t u(\cdot,t) \|_{L^2(0,T,H^{m+1}({\mathcal{D}}))} 
+ \sum_{k=1}^n 
 \deltat{k} \| h^{2-m} u(\cdot,t_k) \|_{H^2({\mathcal{D}})}
\right).
\end{eqnarray*}
\end{theorem}

The proof of this theorem exploits the same techniques as the one of Theorem \ref{final-a-priori-error-estimate}, which is why we will not present it here.

\section{Reformulation of the continuous problem}
\label{section-aux-results-analytical}

In this section, we establish some auxiliary results and preliminaries concerning the model problem (\ref{model-problem}). In particular, we introduce a suitable scalar product on $H^1({\mathcal{D}})$ which can be associated with the operator $L$ and which is more convenient for the analysis in the following sections.

If clear from the context, we subsequently leave out the integration variable in our integrals, for instance we write $\int_{\mathcal{D}} v$ for $\int_{\mathcal{D}} v(\mathbf{x}) \hspace{2pt} d\mathbf{x}$.
In order to analyze problem (\ref{model-problem}) properly, we require some additional definitions and auxiliary results.

\begin{definition}
For any subdomain $\omega \subset {\mathcal{D}}$
we define the sesquilinear form $(\cdot,\cdot)_{E(\omega)}$ by
\begin{eqnarray*}
\lefteqn{(v,w)_{E(\omega)} := \int_{\omega} \left( A^{1/2} \nabla v - \ci 2^{-1} A^{-1/2}
b
v \right) \cdot \overline{\left( A^{1/2} \nabla w - \ci 2^{-1} A^{-1/2}
b
w \right)} }\\
&+&  \int_{\omega} (
c
- (1/4) |A^{-1/2}
b
|^2) v \overline{w} \hspace{180pt}
\end{eqnarray*}
for $v,w \in H^1(\omega)$.
Note that $c - (1/4) |A^{-1/2} b|^2$ is positive by (A5).
Accordingly, we define the norm $\| \cdot \|_{E(\omega)}$ by $\| v \|_{E(\omega)}:=\sqrt{(v,v)_{E(\omega)}}$.
\end{definition}

\begin{lemma}\label{norm-equivalence}
Let $\omega \subset \Omega$ be a subdomain. Under assumptions (A1)-(A5), the sesquilinear form $(\cdot,\cdot)_{E(\omega)}$ is a scalar product on $H^1(\omega)$ and the induced norm $\| v \|_{E(\omega)}$ is equivalent to the standard $H^1$-norm $\| \cdot \|_{H^1(\omega)}$. In particular we have for all $v \in H^1(\omega)$
\begin{align*}
\| v \|_{E(\omega)}^2 \ge (1-\zeta_1^{-1}) \| A^{1/2} \nabla v \|_{L^2(\omega)}^2 + \zeta_0  \| v \|_{L^2(\omega)}^2.
\end{align*}
\end{lemma}

\begin{proof}
Obviously, $(\cdot,\cdot)_{E(\omega)}$ is a symmetric sesquilinear form on $H^1_0(\omega)$. Hence, it only remains to show the existence of constants $c_E$ and $C_E$ such that
\begin{align*}
c_E \| v \|_{H^1(\omega)}^2 \le (v,v)_{E(\omega)} \le C_E  \| v \|_{H^1(\omega)}^2 \qquad \mbox{for all } v \in H^1(\omega).
\end{align*}
The upper bound is straightforward using the boundedness of the coefficients. To verify the lower bound, we first observe with Youngs inequality for any $\epsilon>0$ that
\begin{eqnarray*}
\lefteqn{\int_{\omega} |A^{1/2} \nabla v - \ci 2^{-1} A^{-1/2} b v|^2 }\\
&\ge& \int_{\omega} A \nabla v \cdot \overline{\nabla v} - \int_{\omega} |A^{-1/2} b| |v| |A^{1/2} \nabla v| - \frac{1}{4} \int_{\omega} |A^{-1/2} b|^2 |v|^2 \\
&\ge& \left( 1 - \epsilon^{-1} \right) \int_{\omega} A \nabla v \cdot \overline{\nabla v} - \frac{1+\epsilon}{4} \int_{\omega} |A^{-1/2} b|^2 |v|^2.
\end{eqnarray*}
Hence
\begin{align*}
(v,v)_{E(\omega)}
&\ge \left( 1 - \epsilon^{-1} \right) \int_{\omega} A \nabla v \cdot \overline{\nabla v} - \frac{2+\epsilon}{4} \int_{\omega}  |A^{-1/2} b|^2 |v|^2
+ \int_{\omega} c |v|^2.
\end{align*}
Choosing $\epsilon=\zeta_1$ together with (A5) finishes the result (where we assumed $\zeta_1>1$). Also observe that $\zeta_1 \le 1$ leads to degeneracies.
\end{proof}

\begin{conclusion}\label{ellipticity-of-L}
The differential operator $L$ is uniformly elliptic and continuous on $H^1_0({\mathcal{D}})$. In particular it holds
\begin{align}
\label{equation-E-L}(v,w)_{E({\mathcal{D}})} = \langle L(v),w \rangle \qquad \mbox{for all } v,w \in H^1_0({\mathcal{D}}).
\end{align}
\end{conclusion}
%
Observe that Lemma \ref{norm-equivalence} and Conclusion \ref{ellipticity-of-L} imply that the operator $L$ degenerates for $\zeta_0=0$ and $\zeta_1=1$.

\begin{proof}[Proof of Conclusion \ref{ellipticity-of-L}]
Let $v,w \in H^1_0({\mathcal{D}})$, we observe that
\begin{eqnarray*}
\lefteqn{\int_{\mathcal{D}} \left( A^{1/2} \nabla v - \ci 2^{-1} A^{-1/2} b v \right) \cdot \overline{\left(A^{1/2} \nabla w - \ci 2^{-1} A^{-1/2} b w\right)}}\\
&=& \int_{\mathcal{D}} A \nabla v \cdot \overline{\nabla w}
-
\ci 2^{-1} \left( \int_{\mathcal{D}} v b \cdot \overline{\nabla w} -  \int_{\mathcal{D}} \overline{w} b \cdot \nabla v \right)
  + \frac{1}{4} \int_{\mathcal{D}} |A^{-1/2} b|^2 v \overline{w}\\
 &=& \int_{\mathcal{D}} A \nabla v \cdot \overline{\nabla w}
+ \int_{\mathcal{D}} \overline{w} \ci b \cdot \nabla v
+ \frac{1}{4} \int_{\mathcal{D}} |A^{-1/2} b|^2 v \overline{w},
\end{eqnarray*}
where we used that $\nabla \cdot b = 0$. Hence we have
\begin{align}
\label{eq-lem-1}
(v,w)_{E({\mathcal{D}})}=
\int_{\mathcal{D}} A \nabla v \cdot \overline{\nabla w}
+ \int_{\mathcal{D}} \overline{w} \ci b \cdot \nabla v
+ \int_{\mathcal{D}}c v \overline{w}.
\end{align}
Assumption (A4) finishes the proof of \eqref{equation-E-L}.
The continuity and ellipticity of $L$ hence follow using Lemma \ref{norm-equivalence}.
\end{proof}

\begin{remark}
Let $\omega \subset {\mathcal{D}}$ be a subdomain and $v,w \in H^1(\omega)$ arbitrary. Under assumptions (A1)-(A5)
we see that there exists a constant $C$ (only depending on $A$, $b$ and $c$) such that
$$\left| \int_{\omega} A \nabla v \cdot \overline{\nabla w } + b \nabla v \cdot \overline{w} + c v \overline{w} \right|
\le C \| v \|_{H^1(\omega)}  \| w \|_{H^1(\omega)}.$$
Using the norm equivalence of Lemma \ref{norm-equivalence} we hence also have
\begin{align}
\label{local-continuity}\left| \int_{\omega} A \nabla v \cdot \overline{\nabla w } + b \nabla v \cdot \overline{w} + c v \overline{w} \right|
\le C_E \| v \|_{E(\omega)} \| w \|_{E(\omega)},
\end{align}
with $C_E=C_E(A,b,c)$. However, note that we do not have $(v,w)_{E(\omega)} = \langle L(v),w \rangle_{H^{-1}(\omega),H^1(\omega)}$ for arbitrary $v,w \in H^1(\omega)$.
\end{remark}

\section{Existence and uniqueness of discrete solutions}
\label{section-exst-uni-disc-sol}

In this section we
consider
the existence and uniqueness of discrete solutions. For that, we require the following well known result which can be found e.g. in the book by Thom\'ee \cite[Lemma 6.4]{Tho06}. It can be easily proved using Sobolev embeddings with a inverse inequality. 

\begin{lemma}\label{lemma-thomee}
Let ${\mathcal{D}} \subset \R^d$ be a convex domain. Then there exists some constant $C_{\infty}$ such that for all $v_h\in S_h$
\begin{align*}
\| v_h \|_{L^{\infty}({\mathcal{D}})} \le
C_{\infty} \ell_h
 \| \nabla v_h \|_{L^2({\mathcal{D}})},
\end{align*}
where
\begin{align*}
\ell_h :=
\begin{cases}
 |\ln{\hmin}|^{1/2} \quad &\mbox{for } d=2 \\
 \hspace{3pt}|{\hmin}|^{-1/2} \hspace{3pt} \quad &\mbox{for } d=3.
\end{cases}
\end{align*}
\end{lemma}

We treat the existence of discrete solutions $u_h^n$ of (\ref{cnd-problem}) together with the solutions of some regularized auxiliary problem.
This auxiliary problem is essential for the analysis
of (\ref{cnd-problem}).
For this purpose, we recall a lemma that was basically proved in \cite{KaM98}.
\begin{lemma}\label{makridakis-lemma}
Let $M  \in \R$ be given by
\begin{align}
\label{def-of-M}M:= \| u \|_{W^{1,\infty}(I_n, W^{1,\infty}(\mathcal{D}))} + C_{W^{1,\infty}} (\mbox{\rm diam}(\mathcal{D})+1) \| \nabla u \|_{L^{\infty}(I_n \times \mathcal{D})},
\end{align}
where $C_{W^{1,\infty}}$ is the constant from (A7). Then, there exists a function $f_{M} : \mathbb{C} \rightarrow \mathbb{C}$ and a constant $c_{M}>0$ such that:
\begin{align}
\label{f_M_cond_1}f_M(z) &= |z|^2 z, \qquad \hspace{85pt}\mbox{if } |z|\le M, \\
\label{f_M_cond_2}\langle f_M(z), z \rangle &\in \R_{\ge 0}, \qquad \hspace{89pt}\mbox{for all } z \in \C, \\
\label{f_M_cond_3a}|f_M(z)| &\le 2 M^2 |z|, \qquad \hspace{75pt}\mbox{for all } z \in \C, \\
\label{f_M_cond_3}|f_M(z) - f_M(w)| &\le 10 M^2 |z - w|, \qquad \hspace{49pt}\mbox{for all } z,w \in \C, \\
\label{f_M_cond_4}\energy{ f_M(z) - f_M(w) } &\le c_{M} \energy{ z -w } \\
\nonumber& \hspace{-10pt}\mbox{for all } z,w \in H^1_0({\mathcal{D}}) \mbox{ with } \| w \|_{W^{1,\infty}({\mathcal{D}})} \le M.
\end{align}
\end{lemma}

The above lemma is a slightly generalized version of \cite[Lemma 4.1]{KaM98} in the sense that we are more precise about the constants in (\ref{f_M_cond_3a}) and (\ref{f_M_cond_3}), condition (\ref{f_M_cond_2}) is new and condition (\ref{f_M_cond_4}) is formulated with a different norm. The latter two points are obvious, therefore we only prove (\ref{f_M_cond_3a}) and (\ref{f_M_cond_3}).

\begin{proof}
Let us define $\theta:=M^2$,
$g(s):=3 \theta^{-4} s^5 - 7 \theta^{-3} s^4 + 4 \theta^{-2}s^3 + s$
and the curve $\gamma : \R \rightarrow \R$ by
\begin{align*}
\gamma(s) :=
\begin{cases}
s &\mbox{for } s \le \theta \\
g(s-\theta)+\theta &\mbox{for } s \in [\theta,2\theta] \\
2 \theta &\mbox{for } s \ge 2\theta.
\end{cases}
\end{align*}
It can be verified that $\gamma \in C^2(\R)$ and we can hence define $f_M(z):=\gamma(|z|^2) z$ for $z \in \C$. In order to verify the properties of $f_M$ it is sufficient to check the behavior of $g$ on $[0,\theta]$. It holds
\begin{align*}
g^\prime(s)
&=(15 \theta^{-4} s^2 + 2 \theta^{-3} s + \theta^{-2})(s - \theta)^2
\end{align*}
which is obviously strictly positive on $[0,\theta)$. Hence $g$ is monotonically increasing and so is $\gamma$. Furthermore, we have
\begin{align*}
g^{\prime\prime}(s)&
= 60 \theta^{-4} s (\theta - s) ( 2\theta/5 - s),
\end{align*}
which implies that $g^{\prime}$ has a maximum in $2\theta/5$ with $g^{\prime}(2\theta/5)\le2$. 
We observe $|g^{\prime\prime}(s)|\le 60 \theta^{-1}$. Combining these properties of $g$ allows us to derive (\ref{f_M_cond_3a}) and (\ref{f_M_cond_3}), with the constants as given in the lemma. Property (\ref{f_M_cond_2}) is obvious since $\gamma$ is monotonically increasing and hence non-negative on $[0,\infty)$. Condition (\ref{f_M_cond_4}) is stated in \cite[Lemma 4.1]{KaM98} with the $H^1$-seminorm, but follows directly by the norm equivalence that we showed earlier.
\end{proof}

Using the previously introduced function $f_M$, we can now state the regularized problem. As we will see later, the solution to the regularized problem is
a solution of the discrete problem (\ref{cnd-problem}) for sufficiently small time steps.

\begin{definition}[Discrete auxiliary problem]
Let $f_M$ denote a function with the properties depicted in Lemma \ref{makridakis-lemma}. Furthermore we let $U^{0}=u_{h}^{0} \in S_h$ with $u_{h}^{0}$ being the initial value used for problem (\ref{cnd-problem}). For $n\ge 1$ we let $U^{n} \in S_h$ denote the solution of
\begin{eqnarray}
\label{aux-problem-bwe}\Ltwo{ U^{n} }{v_h} +
\deltat{n} \hspace{2pt}\ci \hspace{2pt} \left( \langle L(U^{n-\frac{1}{2}}), v_h \rangle
 + \Ltwo{\kappa U^{n-\frac{1}{2}} + \beta f_M(U^{n-\frac{1}{2}})}{v_h} \right)
= \Ltwo{ U^{n-1}}{v_h}
\end{eqnarray}
for all $v_h \in S_h$ and where we defined $U^{n-\frac{1}{2}}:=(U^{n}+U^{n-1})/2$.
\end{definition}

In order to show existence
of the solutions of problem (\ref{cnd-problem}) and (\ref{aux-problem-bwe})
we require the following lemma, which is a well-known conclusion from Brouwers fixed point theorem.
\begin{lemma}\label{brouwer-lemma}
Let $N \in \mathbb{N}$ and let $\overline{B_1(0)}:=\{ \boldsymbol{\alpha} \in \C^N| \hspace{2pt} |\boldsymbol{\alpha}|\le 1\}$ denote the closed unit disk in $\C^N$. Then every continuous function $g : \C^N \rightarrow \C^N$ with $\Re\langle g(\boldsymbol{\alpha}), \boldsymbol{\alpha} \rangle \ge 0$ for all $\boldsymbol{\alpha} \in \partial \overline{B_1(0)}$ has a zero in $\overline{B_1(0)}$, i.e. a point $\boldsymbol{\alpha}_0 \in \overline{B_1(0)}$ with $g(\boldsymbol{\alpha}_0)=0$.
\end{lemma}

If there exits no $\boldsymbol{\alpha}_0 \in \overline{B_1(0)}$ with $g(\boldsymbol{\alpha}_0)=0$, then $\hat{g}(\boldsymbol{\alpha}):= - g(\boldsymbol{\alpha})/|g(\boldsymbol{\alpha})|$ (interpreted as a function $\hat{g} : \R^{2N} \rightarrow \R^{2N}$) has a fixed point $\boldsymbol{\alpha}^{\ast}  \in \overline{B_1(0)}$ by Brouwers fixed point theorem. Hence $1=|\hat{g}(\boldsymbol{\alpha}^{\ast})|^2=\langle \hat{g}(\boldsymbol{\alpha}^{\ast}),\boldsymbol{\alpha}^{\ast}\rangle=- \langle g(\boldsymbol{\alpha}^{\ast}),\boldsymbol{\alpha}^{\ast}\rangle/|g(\boldsymbol{\alpha}^{\ast})|=- \Re \langle g(\boldsymbol{\alpha}^{\ast}),\boldsymbol{\alpha}^{\ast}\rangle/|g(\boldsymbol{\alpha}^{\ast})|<0$, which is a contradiction.

\begin{lemma}\label{lemma-exist-uniq-disc-sols}
For every $n\ge1$ there exists a solution 
$U^{n} \in S_h$ of problem (\ref{aux-problem-bwe}). If the time step size is such that
$\deltat{n}<2\left( \| \Im(\kappa) \|_{L^{\infty}({\mathcal{D}})} +\beta{10 M^2}\right)^{-1}$
then the solution $U^{n} \in S_h$ is also unique. Recall that $M$ is the constant appearing in Lemma \ref{makridakis-lemma}. Furthermore, if $\Im(\kappa)=0$ and if $\deltat{n}$ and $h$ are such that $\deltat{n}\hmin^{-d} \rightarrow 0$ for $\deltat{n},h\rightarrow 0$, then the solution $u_h^n \in S_h$ of problem (\ref{cnd-problem}) is unique for sufficiently small $\deltat{n}$ as well.
\end{lemma}

\begin{proof}
We start with the existence result for the solution $U^{n}$ of problem (\ref{aux-problem-bwe}). First, recall that $N_h=$dim$(S_h)$ and that $\lambda_m$ denotes the $m$'th Lagrange basis function. We want to apply Lemma \ref{brouwer-lemma} and define $g : \C^{N_h} \rightarrow \C^{N_h}$ for $\boldsymbol{\alpha}\in \C^{N_h}$ by
\begin{eqnarray*}
\lefteqn{g_\ell(\boldsymbol{\alpha}) := - \deltat{n}^{-1} \hspace{2pt}\ci \sum_{m=1}^{N_h} \boldsymbol{\alpha}_m \Ltwo{ \lambda_m }{ \lambda_\ell }
}\\
&\enspace&
 +
\frac{1}{2} \sum_{m=1}^{N_h} \boldsymbol{\alpha}_m \hspace{2pt} \langle L(\lambda_m), \lambda_\ell \rangle
 + \Ltwo{ (\kappa \mbox{\rm Id} + \beta f_M)( \frac{1}{2}U^{n-1} + \frac{1}{2}\sum_{m=1}^{N_h} \boldsymbol{\alpha}_m \lambda_m )}{\lambda_\ell} +F_\ell,
\end{eqnarray*}
where $F\in \C^{N_h}$ is defined by
\begin{align*}
F_\ell := \frac{1}{2}\langle L(U^{n-1}), \lambda_\ell \rangle + \Ltwo{ \deltat{n}^{-1} \hspace{2pt}\ci U^{n-1}}{\lambda_\ell}.
\end{align*}

To show the existence of some $\boldsymbol{\alpha}_0$ with $g(\boldsymbol{\alpha}_0)=0$, it is sufficient (by scaling arguments) to show that there exists some $K \in \R_{>0}$ so that $\Re\langle g(\boldsymbol{\alpha}), \boldsymbol{\alpha} \rangle \ge 0$ for all $\boldsymbol{\alpha} \in \C^{N_h}$ with $|\boldsymbol{\alpha}|= K$. For brevity, let us denote $\alpha:=\sum_{m=1}^{N_h} \boldsymbol{\alpha}_m \lambda_m$. Since $\Re \left( \deltat{n}^{-1} \hspace{2pt}\ci \langle\alpha,\alpha\rangle_{L^2({\mathcal{D}})} \right) = 0$; $(\alpha,\alpha)_{E({\mathcal{D}})}=  \langle L(\alpha),\alpha \rangle$ and by construction of $f_M$
\begin{align*}
\Re \left( \langle (\kappa \mbox{\rm Id} + \beta f_M)(\frac{\alpha + U^{n-1}}{2}),\frac{\alpha + U^{n-1}}{2}) \rangle_{L^2({\mathcal{D}})} \right)&\ge0,
 \end{align*}
we obtain
\begin{align*}
\Re\langle g(\boldsymbol{\alpha}), \boldsymbol{\alpha} \rangle &\ge
\frac{1}{2}\| \alpha \|_{E({\mathcal{D}})}^2
+ \Re\langle F, \boldsymbol{\alpha} \rangle
- \Ltwo{ (\kappa \mbox{\rm Id} + \beta f_M)( \frac{1}{2}U^{n-1} + \frac{1}{2}\sum_{m=1}^{N_h} \boldsymbol{\alpha}_m \lambda_m)}{U^{n-1}}\\
&\ge \frac{1}{2}\| \alpha \|_{E({\mathcal{D}})}^2
-  \| \deltat{n}^{-1} U^{n-1} \|_{L^2({\mathcal{D}})} \| \alpha \|_{L^2({\mathcal{D}})}
- \frac{1}{2} \| \alpha \|_{E({\mathcal{D}})} \| U^{n-1} \|_{E({\mathcal{D}})}\\
&\enspace \qquad
- \frac{1}{2} \left( \| \kappa\|_{L^{\infty}(\mathcal{D})} + \beta {2M^2} \right) \|U^{n-1} \|_{L^2({\mathcal{D}})}
\left( \| \alpha \|_{L^2({\mathcal{D}})} + \|U^{n-1} \|_{L^2({\mathcal{D}})}\right)
\\
&\ge \| \alpha \|_{E({\mathcal{D}})} \left( C_1 \| \alpha \|_{E({\mathcal{D}})} - C_2 \right) - C_3,
\end{align*}
where we used the Poincar\'e-Friedrichs inequality in the last step and where $C_1$, $C_2$ and $C_3$ are appropriate $\alpha$-independent positive constants.
Consequently, for all $\boldsymbol{\alpha}$ with $\|{\alpha}\|_{E({\mathcal{D}})} \ge C_2/C_1 + \sqrt{ (C_2/C_1)^2 + C_3/C_1}$ we have $\Re\langle g(\boldsymbol{\alpha}), \boldsymbol{\alpha} \rangle \ge 0$ and hence, by norm equivalence in finite dimensional spaces, there exists a sufficiently large $K$ such that
$\Re\langle g(\boldsymbol{\alpha}), \boldsymbol{\alpha} \rangle \ge 0$  for all $\boldsymbol{\alpha}$ with $|\boldsymbol{\alpha}|= K$. This gives us existence of a discrete solution of (\ref{aux-problem-bwe}). 

For uniqueness in (\ref{aux-problem-bwe})
we use an $L^2$-contraction argument. Let us compare two solution $U_{(1)}^n$ and $U_{(2)}^n$ of problem (\ref{aux-problem-bwe}). Using
the equation and testing with $U_{(1)}^n - U_{(2)}^n$
we get
\begin{eqnarray*}
\lefteqn{\| U_{(1)}^n - U_{(2)}^n \|_{L^2({\mathcal{D}})}^2}\\
&=& - \frac{\deltat{n}}{2} \hspace{2pt}\ci \hspace{2pt} \left( \langle L( U_{(1)}^n - U_{(2)}^n) , U_{(1)}^n - U_{(2)}^n \rangle
 + \Ltwo{ \kappa (U_{(1)}^n - U_{(2)}^n)}{ U_{(1)}^n - U_{(2)}^n } \right)\\
 &\enspace& \quad - 2 \deltat{n} \hspace{2pt}\ci \hspace{2pt} \beta \Ltwo{f_M(\frac{U_{(1)}^n + U^{n-1}}{2}) - f_M(\frac{U_{(2)}^n + U^{n-1}}{2})}{\frac{U_{(1)}^n + U^{n-1}}{2} - \frac{U_{(2)}^n + U^{n-1}}{2}}\\
&=&
 \frac{1}{2}\Ltwo{ \deltat{n} \Im(\kappa) (U_{(1)}^n - U_{(2)}^n)}{ U_{(1)}^n - U_{(2)}^n }\\
 &\enspace& \quad + 2 \deltat{n} \hspace{2pt} \beta \hspace{2pt}  \Im \left( \Ltwo{f_M(\frac{U_{(1)}^n + U^{n-1}}{2}) - f_M(\frac{U_{(2)}^n + U^{n-1}}{2})}{\frac{U_{(1)}^n + U^{n-1}}{2} - \frac{U_{(2)}^n + U^{n-1}}{2}} \right) \\
 &\overset{(\ref{f_M_cond_3})}{\le}& \frac{1}{2} \deltat{n} \left( \| \Im(\kappa) \|_{L^{\infty}({\mathcal{D}})} +\beta{10 M^2}\right) \| U_{(1)}^n - U_{(2)}^n \|_{L^2({\mathcal{D}})}^2.
\end{eqnarray*}
Since we assumed that $\deltat{n} \left( \| \Im(\kappa) \|_{L^{\infty}({\mathcal{D}})} +\beta{10 M^2}\right)<2$ we conclude $\| U_{(1)}^n - U_{(2)}^n \|_{L^2({\mathcal{D}})}=0$ and have hence uniqueness.

For two solutions $u_{h,(1)}^{n}$ and $u_{h,(2)}^{n}$ of the original IRK scheme \eqref{cnd-problem}, we
use the additional assumption $\Im(\kappa)=0$ to conclude with the mass conservation that
$$
\| u_{h,(1)}^{n} \|_{L^2(\mathcal{D})}=\| u_{h,(2)}^{n}  \|_{L^2(\mathcal{D})}=\| u_{h}^{n-1} \|_{L^2(\mathcal{D})}=\| u_{h}^{0} \|_{L^2(\mathcal{D})} 
= \| \mathcal{I}_h( u_{0} ) \|_{L^2(\mathcal{D})} 
\le 
C \| u_{0} \|_{H^2(\mathcal{D})},
$$
where we used the stability estimate $\| \mathcal{I}_h(v) \|_{L^2(\mathcal{D})} \le C \| v \|_{H^2(\mathcal{D})}$ for $v\in H^{2}(\mathcal{D})$ for the Lagrange interpolation operator $\mathcal{I}_h$. With this, we
can proceed as before to obtain
\begin{eqnarray*}
\lefteqn{\| u_{h,(1)}^{n} - u_{h,(2)}^{n} \|_{L^2({\mathcal{D}})}^2}\\
 &\le& \frac{\deltat{n}}{8} \hspace{2pt} \beta \hspace{2pt} \Im \Ltwo{\left( \left| u_{h,(1)}^{n} + u_{h}^{n-1}\right|^2 - \left| u_{h,(2)}^{n} + u_{h}^{n-1} \right|^2\right) (u_{h,(2)}^{n} + u_h^{n-1}) }
 { u_{h,(1)}^n - u_{h,(2)}^n } \\
 &\le& \frac{\deltat{n}}{4} \hspace{2pt} \beta \hspace{2pt} \int_{\mathcal{D}}
\left( \left| u_{h,(1)}^{n} + u_{h}^{n-1}\right|^2 + \left| u_{h,(2)}^{n} + u_{h}^{n-1} \right|^2 \right)
 | u_{h,(1)}^n - u_{h,(2)}^n |^2 \\
  &\le& \deltat{n} \hspace{2pt} \beta \hspace{2pt} C 
  \| u_{0} \|_{H^2(\mathcal{D})}^2 \| u_{h,(1)}^{n} - u_{h,(2)}^{n} \|_{L^{\infty}({\mathcal{D}})}^2.
\end{eqnarray*}
With the inverse estimate $\|  u_{h,(1)}^{n} - u_{h,(2)}^{n} \|_{L^{\infty}(\mathcal{D})} \le C \hmin^{ - d/2 }  \|  u_{h,(1)}^{n} - u_{h,(2)}^{n} \|_{L^2(\mathcal{D})}$ we conclude that for an appropriate positive constant
$$
\| u_{h,(1)}^{n} - u_{h,(2)}^{n} \|_{L^2({\mathcal{D}})}^2 \le C(\beta,\|u_0\|_{H^2(\mathcal{D})}) \hspace{3pt}
\deltat{n} \hmin^{ - d } \| u_{h,(1)}^{n} - u_{h,(2)}^{n} \|_{L^2({\mathcal{D}})}^2
$$
and hence $u_{h,(1)}^{n} = u_{h,(2)}^{n}$ for sufficiently small $\deltat{n}$.
\end{proof}

\section{A priori error estimates}
\label{final-apriori-proofs}

In the following we assume that $u$ denotes a solution of (\ref{model-problem}) with sufficient regularity.
In this section we derive an a priori error estimate for the discrete solutions. However, instead of taking (\ref{cnd-problem}) as our reference problem we follow the ideas of \cite{KaM98} and take the auxiliary problem (\ref{aux-problem-bwe}) as our reference. In this context, note that by the definitions of $u$ and $f_M$ we have
\begin{align}
\label{exact-solution-with-discrete-test-func}\Ltwo{ u(\cdot,t_n) }{ v_h } +\ci \int_{I_n} \langle L(u),v_h \rangle +\ci \int_{I_n} \Ltwo{ \kappa u + \beta f_M(u) }{ v_h }
= \Ltwo{ u(\cdot,t_{n-1}) }{ v_h }
\end{align}
for all $v_h \in S_h$.
Since $u$ is continuous in time we can define $u^n:=u(\cdot,t_n)$.

For simplicity (and slightly abusing the notation), we write for $v \in H^2({\mathcal{D}})\cap H^1_0({\mathcal{D}})$
$$\triangleL v := - \nabla \cdot \left( A \nabla v \right) + b \cdot \nabla v + c v,$$
so that
\begin{align}
\label{green-for-nablaL}
(v,w)_{E({\mathcal{D}})} = \Ltwo{\triangleL v}{w} \qquad \mbox{for all } v,w \in H^1_0({\mathcal{D}}) \cap H^2({\mathcal{D}}).
\end{align}

In order to derive the a priori error estimates, we first derive an error identity and then estimate the various terms in the identity.

Before starting, recall Definition \ref{def-ritz-projection}, i.e. the definition of the Ritz-projection associated with $L$.
Note that we do not include the term $(\kappa v , \phi_h )$ in the Ritz-projection since we want $L$ to be a smooth and self-adjoint operator. Since $\kappa$ can be imaginary, equation (\ref{equation-E-L}) would not be valid any longer.

Finally, we also recall a standard result (which follows from the best approximation property of $P_h$ with respect to the $H^1$-norm and an Aubin-Nitsche duality argument).

\begin{lemma}
Assume (A1)-(A5). There exist generic positive constants $C_1$ and $C_2$ such that
\begin{align}
\label{estimate-ritz-projection}\| v - \PR(v) \|_{L^2({\mathcal{D}})} \le C_1 | h^2 v |_{H^2({\mathcal{D}})} \qquad \mbox{and} \qquad
\| v - \PR(v)\|_{E({\mathcal{D}})} \le C_2 | h v |_{H^2({\mathcal{D}})}
\end{align}
for all $v\in H^2({\mathcal{D}}) \cap H^1_0({\mathcal{D}})$.
\end{lemma}

In the first step, we establish an error identity.
\begin{lemma}[Error identity]
We introduce the abbreviation $\hat{f}(v):=\kappa v + \beta f_M(v)$. For $n \in \N$, $n\ge1$ we define the error splitting by
\begin{align}
\label{error-splitting}
e_h^n := \underset{\mbox{\large$=:E_h^n \in S_h$}}{\underbrace{(U^n_h - \PR(u^n))}} +
(\PR(u^n) -  u^n)
\end{align}
and the error contributions by
\begin{align*}
\xi_n^{(1)} &:= \int_{I_n}  \PR(\partial_t u(\cdot,t)) - \partial_t u(\cdot,t) \hspace{2pt} dt, \qquad
\xi_n^{(2)} := \deltat{n} \hspace{2pt} \ci \hspace{2pt} \left( \hat{f}( {\frac{\PR(u^{n}) +\PR(u^{n-1})}{2}} ) - \hat{f}({\frac{U^{n}+U^{n-1}}{2}}) \right),\\
\xi_n^{(3)} &:= \ci \int_{I_n} \hat{f}(u(\cdot,t)) - \hat{f}( \frac{\PR(u^{n}) +\PR(u^{n-1})}{2} ) \hspace{2pt} dt, \qquad \xi_n^{(4)} := \ci \int_{I_n} u(\cdot,t) - \frac{u^{n}+u^{n-1}}{2} \hspace{2pt} dt.
\end{align*}
With these notations the following $L^2$-norm identity holds for $E_h^{n}$
\begin{eqnarray}
\label{error-identity}\lefteqn{\| E_h^{n} \|_{L^2({\mathcal{D}})}^2}\\
\nonumber&=& \| E_h^{n-1} \|_{L^2({\mathcal{D}})}^2 +
\Re \left( \Ltwo{ \xi_n^{(1)} + \xi_n^{(2)} + \xi_n^{(3)} }{ E_h^{n}+ E_h^{n-1} } + \langle L(\xi_n^{(4)}) , E_h^{n}+ E_h^{n-1} \rangle  \right).
\end{eqnarray}
and the following energy-nom identity
\begin{eqnarray}
\nonumber\lefteqn{\energy{ E_h^{n}}^2 = \energy{ E_h^{n-1}}^2 + \Re \langle L( E_h^{n} + E_h^{n-1}), P_{L^2}( \xi_n^{(1)} + \xi_n^{(2)} + \xi_n^{(3)} ) \rangle}\\
\label{error-identity-enrgy}&+& \Re \left( \ci \langle L( E_h^{n} + E_h^{n-1} ), \int_{I_n} P_{L^2}( \triangleL u(\cdot,t) ) - \frac{P_{L^2} (\triangleL u^{n})+ P_{L^2} (\triangleL u^{n-1})}{2} \hspace{2pt} dt
\rangle \right)
\end{eqnarray}
\end{lemma}

\begin{proof}
Recalling the definition of $U^{n}$ we have for all $v_h \in S_h$
\begin{eqnarray}
\label{identity-proof-1}\Ltwo{ U^{n} -  U^{n-1} }{v_h} +
\deltat{n} \hspace{2pt} \ci \hspace{2pt} \langle L(\frac{U^{n}+U^{n-1}}{2}), v_h \rangle
= - \deltat{n} \hspace{2pt} \ci \hspace{2pt} \Ltwo{ \hat{f}(\frac{U^{n}+U^{n-1}}{2})}{v_h}.
\end{eqnarray}
Subtracting the term
\begin{eqnarray*}
\Ltwo{ \PR(u^{n}) - \PR(u^{n-1}) }{v_h} +
\deltat{n} \hspace{2pt} \ci \hspace{2pt} \langle L( \frac{\PR(u^{n}) +\PR(u^{n-1})}{2} ) , v_h \rangle
\end{eqnarray*}
on both sides of (\ref{identity-proof-1}) gives us
\begin{eqnarray}
\nonumber\lefteqn{\Ltwo{ E_h^{n} -E_h^{n-1} }{v_h} +
 \frac{1}{2}\deltat{n} \ci \hspace{2pt} \langle L(E_h^{n} + E_h^{n-1}), v_h \rangle}\\
\nonumber &=& \Ltwo{ \PR(u^{n-1})}{v_h} - \Ltwo{ \PR(u^{n})}{v_h}
-\deltat{n} \hspace{2pt} \ci \hspace{2pt} \langle L( \frac{\PR(u^{n}) +\PR(u^{n-1})}{2} ) , v_h \rangle\\
\nonumber&\enspace& \quad - \deltat{n} \hspace{2pt} \ci \hspace{2pt} \Ltwo{ \hat{f}(\frac{U^{n}+U^{n-1}}{2})}{v_h} \\
\nonumber&=&
\Ltwo{ \PR(u^{n-1})}{v_h} - \Ltwo{ \PR(u^{n})}{v_h}
-\deltat{n} \hspace{2pt} \ci \hspace{2pt} \langle L( \frac{u^{n} +u^{n-1}}{2} ) , v_h \rangle\\
\nonumber&\enspace& \quad
+ \deltat{n} \hspace{2pt} \ci \hspace{2pt}  \left( \Ltwo{
 \hat{f}( {\frac{\PR(u^{n}) +\PR(u^{n-1})}{2}} ) - \hat{f}({\frac{U^{n}+U^{n-1}}{2}}) - \hat{f}( \frac{\PR(u^{n}) +\PR(u^{n-1})}{2} )}{v_h} \right)\\
\label{error-identity-with-test-func}&\overset{(\ref{exact-solution-with-discrete-test-func})}{=}&
  \Ltwo{\PR(u^{n-1}) -  u^{n-1}}{v_h} - \Ltwo{\PR(u^{n}) -  u^{n} }{v_h}\\
\nonumber&\enspace& \quad
+ \ci \langle \int_{I_n} L(u)(\cdot,t) \hspace{2pt} dt  - \deltat{n} L(\frac{u^{n}+u^{n-1}}{2}) ,v_h \rangle\\
\nonumber&\enspace& \quad+ \ci \langle \int_{I_n} \hat{f}(u(\cdot,t)) \hspace{2pt} dt - \deltat{n} \hat{f}( \frac{\PR(u^{n}) +\PR(u^{n-1})}{2} ), v_h \rangle_{L^2({\mathcal{D}})}\\
\nonumber&\enspace& \quad+ \deltat{n} \hspace{2pt} \ci \hspace{2pt}  \left(  \Ltwo{
 \hat{f}( {\frac{\PR(u^{n}) +\PR(u^{n-1})}{2}} ) - \hat{f}({\frac{U^{n}+U^{n-1}}{2}})}{v_h} \right)
\end{eqnarray}
Testing with $v_h=E_h^{n}+ E_h^{n-1}$ and only using the real part of the equation gives us
\begin{eqnarray*}
\lefteqn{ \| E_h^{n} \|^2_{L^2({\mathcal{D}})}}\\
&=& \deltat{n} \hspace{2pt} \Re\left( \ci \hspace{2pt}
\Ltwo{ 
 \hat{f}( {\frac{\PR(u^{n}) +\PR(u^{n-1})}{2}} ) - \hat{f}({\frac{U^{n}+U^{n-1}}{2}})
}{ E_h^{n}+ E_h^{n-1} } \right)
+ \| E_h^{n-1} \|^2_{L^2({\mathcal{D}})} \\
&\enspace& \enspace + \Re\left( \Ltwo{\PR(u^{n-1}) -  u^{n-1}}{ E_h^{n}+ E_h^{n-1} } \right) - \Re\left( \Ltwo{ \PR(u^{n}) -  u^{n} }{ E_h^{n}+ E_h^{n-1} } \right) \\
&\enspace& \enspace
+ \Re\left( \ci \langle \int_{I_n} L(u)(\cdot,t) \hspace{2pt} dt  - \deltat{n} L( \frac{u^{n}+u^{n-1}}{2} ) , E_h^{n}+ E_h^{n-1} \rangle \right) \\
&\enspace& \enspace
+ \Re \left(  \ci  \langle \int_{I_n} \hat{f}(u(\cdot,t)) \hspace{2pt} dt - \deltat{n} \hat{f}( \frac{\PR(u^{n}) +\PR(u^{n-1})}{2} ), E_h^{n}+ E_h^{n-1} \rangle_{L^2({\mathcal{D}})} \right).
\end{eqnarray*}
The simplification
\begin{eqnarray*}
\lefteqn{\Re\left( \Ltwo{\PR(u^{n-1}) -  u^{n-1}}{ E_h^{n}+ E_h^{n-1} } \right) - \Re\left( \Ltwo{ \PR(u^{n}) -  u^{n} }{ E_h^{n}+ E_h^{n-1} } \right)}\\
&=& \Re\left( \Ltwo{\int_{I_n}  \PR(\partial_t u(\cdot,t)) - \partial_t u(\cdot,t) \hspace{2pt} dt   }{ E_h^{n}+ E_h^{n-1} } \right)
\end{eqnarray*}
finishes the proof of the $L^2$-norm identity.

To derive the energy-norm identity we use the $L^2$-Riesz representer $G_h^n\in S_h$ of the error functional $\langle L(E_h^n) , \cdot \rangle$. The Riesz representer $G_h^n\in S_h$ is characterized by the equation
\begin{align}
\Ltwo{ v_h }{ G_h^n } = \langle L( E_h^{n}+ E_h^{n-1} ) , v_h \rangle \qquad \mbox{for all } v_h \in S_h.
\end{align}
Testing with $v_h = G_h^n$ in (\ref{error-identity-with-test-func}) and using $\langle L(E_h^{n}+ E_h^{n-1}), G_h^n \rangle = \| G_h^n \|_{L^2({\mathcal{D}})}^2$ we obtain
\begin{eqnarray}
\nonumber\lefteqn{ \langle L(E_h^{n} + E_h^{n-1}), E_h^{n}- E_h^{n-1} \rangle  +
 \deltat{n} \ci \hspace{2pt} \frac{1}{2} \| G_h^n \|_{L^2({\mathcal{D}})}^2 } \\
\label{energy-identity-proof-step-1}&=&
 \Ltwo{\PR(u^{n-1}) -  u^{n-1}}{ G_h^n } - \Ltwo{\PR(u^{n}) -  u^{n} }{ G_h^n }\\
\nonumber&\enspace& \quad
+ \ci \int_{I_n} \langle L(u)(\cdot,t) - L( \frac{u^{n}+u^{n-1}}{2} ) ,G_h^n \rangle \hspace{2pt} dt
\\
\nonumber&\enspace& \quad
+ \ci \langle \int_{I_n} \hat{f}(u(\cdot,t)) \hspace{2pt} dt - \deltat{n} \hat{f}( \frac{\PR(u^{n}) +\PR(u^{n-1})}{2} ), G_h^n \rangle_{L^2({\mathcal{D}})}\\
\nonumber&\enspace& \quad
+ \deltat{n} \hspace{2pt} \ci \hspace{2pt}
 \Ltwo{  \hat{f}( {\frac{\PR(u^{n}) +\PR(u^{n-1})}{2}} ) - \hat{f}({\frac{U^{n}+U^{n-1}}{2}}) }{G_h^n}.
\hspace{100pt}
\end{eqnarray}
Using that $\Re\langle L(E_h^{n} + E_h^{n-1}), E_h^{n}- E_h^{n-1} \rangle=\energy{ E_h^{n}}^2 - \energy{ E_h^{n-1}}^2$ and that
%
%
%
$$\Re \Ltwo{ v }{ G_h^n } = \Re \langle L( E_h^{n}+ E_h^{n-1} ) , P_{L^2}(v) \rangle \qquad \mbox{for all } v \in H^1_0({\mathcal{D}})$$
and taking the real part of equation (\ref{energy-identity-proof-step-1}) yields
\begin{eqnarray}
\nonumber\lefteqn{\energy{ E_h^{n}}^2 = \energy{ E_h^{n-1}}^2 + \Re \langle L( E_h^{n} + E_h^{n-1} ), P_{L^2}( \xi_n^{(1)} + \xi_n^{(2)} + \xi_n^{(3)} ) \rangle}\\
\nonumber&+& \Re \left( \ci \langle L( E_h^{n}+ E_h^{n-1} ), \int_{I_n} P_{L^2}( \triangleL u(\cdot,t) ) -   \frac{P_{L^2} (\triangleL u^{n})+ P_{L^2} (\triangleL u^{n-1})}{2}  \hspace{2pt} dt
\rangle \right)
\end{eqnarray}
and finishes the proof.
\end{proof}

The next lemma is central for estimating the $\hat{f}$-terms in the error identities.
\begin{lemma}
Recall the constant $M$ from \eqref{def-of-M} 
and let
$f(z):=|z|^2 z$. It holds (a.e. in $\mathcal{D}$)
\begin{align}\label{estimate-xi-3-II-step-1-NEW}
\left| \int_{t_{n-1}}^{t_n} f(u(\cdot,t)) \hspace{2pt} dt - \deltat{n}  f\left( \frac{u^{n}+u^{n-1}}{2} \right)\right|
\le \deltat{n} M \left( \frac{3}{4} |u^{n}-u^{n-1}|^2 + \deltat{n}^2 M \| u \|_{W^{2,\infty}(I_n)} \right).
\end{align}
and
\begin{eqnarray}\label{estimate-xi-3-II-step-1-grad-NEW}
\lefteqn{\left| \nabla \left(\int_{t_{n-1}}^{t_n} f(u(\cdot,t)) \hspace{2pt} dt - \deltat{n}  f\left( \frac{u^{n}+u^{n-1}}{2} \right) \right)\right|}\\
\nonumber&\le& 4 \deltat{n} M \left(|u^{n}-u^{n-1}|^2 + |\nabla u^{n}-\nabla u^{n-1}|^2\right)  + \deltat{n}^3 M^2 \left( \| u \|_{W^{2,\infty}(I_n)} + \| \nabla u \|_{W^{2,\infty}(I_n)} \right).
\end{eqnarray}
\end{lemma}

\begin{proof}
We decompose the error under considerations into
\begin{eqnarray}
\label{errorsplit-for-quad}\lefteqn{ \left(\int_{t_{n-1}}^{t_n} f(u(\cdot,t)) \hspace{2pt} dt - \deltat{n}  f\left( \frac{u^{n}+u^{n-1}}{2} \right) \right) }\\
\nonumber&=&\left( \int_{t_{n-1}}^{t_n} f(u(x,t)) \hspace{2pt} dt - \deltat{n} \frac{f(u^n) + f(u^{n-1})}{2} \right) + \deltat{n} \left( \frac{f(u^{n}) +f(u^{n-1})}{2}  - f\left( \frac{u^{n}+u^{n-1}}{2} \right)\right).
\end{eqnarray}
With $f(u)=|u|^2u$, the first term in \eqref{errorsplit-for-quad} can be estimate using the trapezoidal-rule to obtain
\begin{align}
\nonumber\left| \int_{t_{n-1}}^{t_n} |u|^2 u - \deltat{n} \frac{|u^n|^2 u^n + |u^{n-1}|^2 u^{n-1}}{2} \right| 
&\le \frac{1}{4} \deltat{n}^3 \| u \|_{L^{\infty}(I_n)} \left( 2 \| \partial_t u \|_{L^{\infty}(I_n)}^2
+ \| \partial_{tt} u \|_{L^{\infty}(I_n)} \| u \|_{L^{\infty}(I_n)} \right) \\
\label{estimate-xi-3-II-step-1-trap}&\le \deltat{n}^3 M^2 \| u \|_{W^{2,\infty}(I_n)}.
\end{align}
For the second term in \eqref{errorsplit-for-quad}, let $\zeta_n: [0,1] \rightarrow [u^{n-1},u^{n}]$ denote the complex valued (linear) curve given by
$\zeta_n(s):=(1-s) u^{n-1} + s u^{n}$ for $s \in [0,1]$. We have $\zeta_n^{\prime}(z)=u^{n}-u^{n-1}$ (and $\zeta_n^{\prime\prime}=0$). With that, we get with the trapezoidal-rule and the midpoint rule that
\begin{eqnarray}
\nonumber\lefteqn{\left| \frac{f(u^{n}) +f(u^{n-1})}{2}  - f\left( \frac{u^{n}+u^{n-1}}{2} \right)\right|}\\
\nonumber&\le& \left|  \frac{(f\circ\zeta_n)(0) +(f\circ\zeta_n)(1)}{2}  
- \int_{0}^{1} (f \circ \zeta_n)(s) \hspace{2pt} ds 
\right| + 
\left|  \int_{0}^{1} (f \circ \zeta_n)(s) \hspace{2pt} ds  - f\left( \frac{\zeta_n(0)+\zeta_n(1)}{2} \right) \right|\\
\nonumber&\le& 
\frac{1}{12} \| (f \circ \zeta_n)^{\prime \prime} \|_{L^{\infty}(0,1)} + \frac{1}{24} \| (f \circ \zeta_n)^{\prime \prime} \|_{L^{\infty}(0,1)} =  \frac{1}{8} \| (f \circ \zeta_n)^{\prime \prime} \|_{L^{\infty}(0,1)}\\
\label{estimate-xi-3-II-step-1-prim}&\le& 
\frac{3}{4} |u^{n}-u^{n-1}|^2  \| \zeta_n \|_{L^{\infty}(0,1)} 
\le \frac{3}{4} |u^{n}-u^{n-1}|^2  \| u \|_{L^{\infty}(I_n)}
\end{eqnarray}
Combining the estimates \eqref{estimate-xi-3-II-step-1-trap} and \eqref{estimate-xi-3-II-step-1-prim} with 
\eqref{errorsplit-for-quad} finishes the proof of \eqref{estimate-xi-3-II-step-1-NEW}. Estimate \eqref{estimate-xi-3-II-step-1-trap} can be derived analogously by applying trapezoidal-rule and midpoint rule to the function 
$g(s):=2 |(1-s) u^{n-1} + s u^{n}|^2 ((1-s) \nabla u^{n-1} + s \nabla u^{n}) + ((1-s) u^{n-1} + s u^{n})^2 \overline{(1-s) \nabla u^{n-1} + s \nabla u^{n}}$.
\end{proof}

\begin{lemma}[$L^2$-error estimate for $E_h^{n}$]
\label{L2-estimate-step-n}
Consider $n\ge1$ and $E_h^{n}=U^{n} - \PR(u^{n})$.
Let $M$ denote the constant in Lemma \ref{makridakis-lemma}.
There exists a constant $C_M$ that only depends on $M$, $\mathcal{D}$, $\kappa$, $\beta$, $C_{W^{1,\infty}}$ and $C_1$ (cf. the $L^2$-estimate \eqref{estimate-ritz-projection}) such that for all $\deltat{n} < (2C_M)^{-1}$ it holds
\begin{eqnarray}
\label{prelim-error-estimate}
\lefteqn{\| E_h^{n} \|_{L^2({\mathcal{D}})}^2
\le \frac{(1+ C_M \deltat{n} )}{(1- C_M \deltat{n} ) }\| E_h^{n-1} \|_{L^2({\mathcal{D}})}^2 + C_M \| h^2 \partial_t u \|_{L^2(I_n,H^2({\mathcal{D}}))}^2}\\
\nonumber&\enspace& \quad +
C_M \deltat{n} \left( \| \deltat{}^2 \triangleL(\partial_{tt} u) \|_{L^{\infty}(I_n,L^2({\mathcal{D}}))}^2
+  \| \deltat{}^2 u  \|_{W^{2,\infty}(I_n,L^2(\mathcal{D}))}^2 + \| h^2 u \|_{L^{\infty}(I_n,H^2({\mathcal{D}}))}^2 \right).
\end{eqnarray}
\end{lemma}

\begin{proof}
In the following $C_M$ denotes any constant that depends generically on $M$, $\mathcal{D}$, $\kappa$, $\beta$, $C_{W^{1,\infty}}$ and $C_1$.
We estimate the terms on the right of side of the error identity (\ref{error-splitting}) and start with $\Re \left( \Ltwo{ \xi_n^{(1)} }{ E_h^{n}+ E_h^{n-1} } \right)$. We obtain
\begin{eqnarray}
\nonumber\lefteqn{\frac{| \Re \left( \Ltwo{ \xi_n^{(1)} }{ E_h^{n}+ E_h^{n-1} } \right) |}{ \| E_h^{n}+ E_h^{n-1} \|_{L^2({\mathcal{D}})}} \le \| \xi_n^{(1)} \|_{L^2({\mathcal{D}})}}\\
\nonumber&=&  \| \int_{I_n} \PR( \partial_t u(\cdot,t)) -  \partial_t u(\cdot,t) \hspace{2pt} dt \|_{L^2({\mathcal{D}})}
\le \int_{I_n} \| \PR( \partial_t u(\cdot,t)) -  \partial_t u(\cdot,t) \|_{L^2({\mathcal{D}})}  \hspace{2pt} dt \\
\label{est-proof-step-1}&\overset{(\ref{estimate-ritz-projection})}{\le}& C_1 \int_{I_n} \| h^2 \partial_t u(\cdot,t)) \|_{H^2({\mathcal{D}})} \hspace{2pt} dt
\le C_1 \deltat{n}^{1/2} \| h^2 \partial_t u \|_{L^2(I_n,H^2({\mathcal{D}}))}.
\end{eqnarray}
Hence
\begin{align*}
| \Re \left( \Ltwo{ \xi_n^{(1)} }{ {E_h^{n}+ E_h^{n-1}} } \right) | \le
2 \deltat{n} \left( \| E_h^{n} \|_{L^2({\mathcal{D}})}^2 + \| E_h^{n-1} \|_{L^2({\mathcal{D}})}^2 \right) + \frac{1}{4} C_1^2 \| h^2 \partial_t u \|_{L^2(I_n,H^2({\mathcal{D}}))}^2.
\end{align*}

Next we bound the term depending on $\xi_n^{(2)} =  \deltat{n} \hspace{2pt} \ci \hspace{2pt} \left( \hat{f}( {\frac{\PR(u^{n}) +\PR(u^{n-1})}{2}} ) - \hat{f}({\frac{U^{n}+U^{n-1}}{2}}) \right)$. Recalling that $E_h^{n}=U^{n} - \PR(u^{n})$ we obtain
\begin{align}
\nonumber| \Re \left( \Ltwo{ \xi_n^{(2)} }{ E_h^{n}+ E_h^{n-1} } \right) | &\le  \deltat{n}  \|
\hat{f}( {\frac{\PR(u^{n}) +\PR(u^{n-1})}{2}} ) - \hat{f}({\frac{U^{n}+U^{n-1}}{2}}) \|_{L^2({\mathcal{D}})} \| E_h^{n}+ E_h^{n-1} \|_{L^2({\mathcal{D}})}\\
\label{est-proof-step-2}&\overset{(\ref{f_M_cond_3})}{\le} \left( \| \kappa \|_{L^{\infty}({\mathcal{D}})} + {10 M^2} \beta\right) \deltat{n} \left( \| E_h^{n} \|_{L^2({\mathcal{D}})}^2 + \| E_h^{n-1} \|_{L^2({\mathcal{D}})}^2 \right).
\end{align}
Recall $\hat{f}(z)=\kappa z + \beta f_M(z)$. In order to treat $\xi_n^{(3)}=\ci ( \hat{f}(u) - \hat{f}\left( \frac{1}{2} (\PR(u^{n}) +\PR(u^{n-1}) \right) ,1)_{L^2(I_n)}$, we use that $f_{M}(z)=|z|^2 z$ for $|z|\le M$ and the facts that
$\| u \|_{L^{\infty}(I_n \times \mathcal{D})} \le M$ and
$\| P_h(u^n) \|_{L^{\infty}(\mathcal{D})} \le C_{W^{1,\infty}} \mbox{diam}(\mathcal{D}) \| \nabla u^n \|_{L^{\infty}(\mathcal{D})} \le M$ to conclude that
\begin{eqnarray}
\nonumber\lefteqn{\frac{| \Re \left( \Ltwo{ \xi_n^{(3)} }{ E_h^{n}+ E_h^{n-1} } \right) |}{ \| E_h^{n}+ E_h^{n-1} \|_{L^2({\mathcal{D}})}} \le \| \kappa\|_{L^{\infty}(\mathcal{D})} \| \int_{I_n} u(\cdot,t) - \frac{\PR(u^{n}) +\PR(u^{n-1})}{2} \hspace{2pt} dt
 \|_{L^2({\mathcal{D}})}}\\
\label{L2-bound-new-step-fM} &\enspace& \quad  + \beta  \|\int_{I_n} f(u(\cdot,t)) - f(\frac{\PR(u^{n}) +\PR(u^{n-1})}{2}) \hspace{2pt} dt  \|_{L^2({\mathcal{D}})},\hspace{100pt}
\end{eqnarray}
where $f(z):=|z|^2z$. To estimate this, we decompose $\|\int_{I_n} f(u(\cdot,t)) - f(\frac{\PR(u^{n}) +\PR(u^{n-1})}{2}) \hspace{2pt} dt  \|_{L^2({\mathcal{D}})}$ into
\begin{eqnarray*}
\underset{ \mbox{I}_{\xi_n^{(3)}} }{\underbrace{\| \int_{I_n} f(u(\cdot,t)) -  f\left( \frac{u^{n}+u^{n-1}}{2} \right) \hspace{2pt} dt \|_{L^2({\mathcal{D}})}}} + \underset{ \mbox{II}_{\xi_n^{(3)}} }{\underbrace{ \deltat{n} \| f\left( \frac{u^{n}+u^{n-1}}{2} \right) - f\left( \frac{\PR(u^{n}) +\PR(u^{n-1})}{2}\right) \|_{L^2({\mathcal{D}})}}}. 
\end{eqnarray*}
For the first term we use \eqref{estimate-xi-3-II-step-1-NEW} to get
\begin{align*}
|\mbox{I}_{\xi_n^{(3)}}| &\le \deltat{n}^3 M^2 \| u \|_{W^{2,\infty}(I_n,L^2(\mathcal{D}))}
+  \deltat{n} \frac{3}{4} M \| u^{n}-u^{n-1}  \|_{L^4(\mathcal{D})}^2\\
&=  \deltat{n}^3 M^2 \| u \|_{W^{2,\infty}(I_n,L^2(\mathcal{D}))} + \deltat{n} \frac{3}{4} M \| \int_{I_n} \partial_t u(\cdot,t) \hspace{2pt} dt \|_{L^4(\mathcal{D})}^2 \\
&\le
 \deltat{n}^3 M^2 \| u \|_{W^{2,\infty}(I_n,L^2(\mathcal{D}))} + \deltat{n}^3 C_M \| u  \|_{W^{1,\infty}(I_n,L^2(\mathcal{D}))}.
\end{align*}
For the term $\mbox{II}_{\xi_n^{(3)}}$ we get in the usual manner
\begin{eqnarray*}
|\mbox{III}_{\xi_n^{(3)}}| = \deltat{n} \| f\left( \frac{u^{n}+u^{n-1}}{2} \right) - f\left( \frac{\PR(u^{n}) +\PR(u^{n-1})}{2}\right) \|_{L^2({\mathcal{D}})}
&\overset{(\ref{estimate-ritz-projection})}{\le}&
\deltat{n} C_M \| h^2 u \|_{L^{\infty}{(I_n,H^2({\mathcal{D}}))}}.
\end{eqnarray*}
Combining the estimates for $\mbox{I}_{\xi_n^{(3)}}$ and $\mbox{II}_{\xi_n^{(3)}}$ with \eqref{L2-bound-new-step-fM} yields
\begin{eqnarray*}
\frac{| \Re \left( \Ltwo{ \xi_n^{(3)} }{ E_h^{n}+ E_h^{n-1} } \right) |}{ \| E_h^{n}+ E_h^{n-1} \|_{L^2({\mathcal{D}})}}
&\le& C_{M}  \deltat{n}  \left(  \| \deltat{}^2 u  \|_{W^{2,\infty}(I_n,L^2(\mathcal{D}))} + \| h^2 u \|_{L^{\infty}{(I_n,H^2({\mathcal{D}}))}} \right)
\end{eqnarray*}
and hence the final estimate for the $\xi_n^{(3)}$-term
\begin{eqnarray}
\label{est-proof-step-3}
\lefteqn{| \Re \left( \Ltwo{ \xi_n^{(3)} }{ E_h^{n}+ E_h^{n-1} } \right) |}\\
&\le&
\nonumber \deltat{n} \left( \| E_h^{n} \|_{L^2({\mathcal{D}})}^2 + \| E_h^{n-1} \|_{L^2({\mathcal{D}})}^2 \right) +
\deltat{n} C_M \left(  \| \deltat{}^2 u  \|_{W^{2,\infty}(I_n,L^2(\mathcal{D}))} + \| h^2 u \|_{L^{\infty}{(I_n,H^2({\mathcal{D}}))}} \right)^2.
\end{eqnarray}
Next, we bound the term $\langle L(\xi_n^{(4)}) , E_h^{n} \rangle$. It holds
\begin{eqnarray}
\label{est-proof-step-4}
\lefteqn{
\frac{| \Re \langle L(\xi_n^{(4)}) , E_h^{n}+ E_h^{n-1} \rangle |}{ \| E_h^{n}+ E_h^{n-1} \|_{L^2({\mathcal{D}})}}
\overset{(\ref{green-for-nablaL})}{\le}
 \| \triangleL( \xi_n^{(4)} ) \|_{L^2({\mathcal{D}})}
}\\
\nonumber&=& \| \int_{I_n} \triangleL( \frac{u^{n}+u^{n-1}}{2} - u(\cdot,t) ) \hspace{2pt} dt \|_{L^2({\mathcal{D}})}
\le \frac{1}{12} \deltat{n}^3  \| \triangleL(\partial_{tt} u) \|_{L^{\infty}(I_n,L^2({\mathcal{D}}))}.
\end{eqnarray}
Combining the estimates (\ref{est-proof-step-1})-(\ref{est-proof-step-4}) with the error identity (\ref{error-identity})
proves the lemma.
\end{proof}

Recall that according to Lemma \ref{norm-equivalence} and the Poincar\'e-Friedrichs inequality there exist positive constants $c_E$ and $C_E$ such that
\begin{align}
\label{constants-in-norm-equivalence}c_E \| \nabla v \|_{L^2({\mathcal{D}})}^2 \le \| v \|_{E({\mathcal{D}})}^2 \le C_E  \| \nabla v \|_{L^2({\mathcal{D}})}^2 \qquad \mbox{for all } v \in H^1_0({\mathcal{D}}).
\end{align}

\begin{lemma}[Energy-error estimate for $E_h^{n}$]
\label{energy-estimate-step-n}
Consider $n\ge1$ and $E_h^{n}=U^{n} - \PR(u^{n})$.
Let $M$ denote the constant in Lemma \ref{makridakis-lemma}.
There exists a constant $C_M$ that only depends on $M$, $\mathcal{D}$, $d$, the data $L$, $\kappa$, $\beta$, the norm-equivalence constants $C_E$ and $c_E$ and the stability constants $C_{W^{1,\infty}}$, $C_{L^2}$ and $C_1$ such that for all $\deltat{n} < (2C_M)^{-1}$ it holds
\begin{eqnarray}
\label{error-enrgy-estimate-prelim}
\lefteqn{\energy{ E_h^{n}}^2
\le \frac{(1+ C_M \deltat{n} )}{(1- C_M \deltat{n} ) }\energy{ E_h^{n-1}}^2 + C_M \| h \partial_t u \|_{L^2(I_n,H^2({\mathcal{D}}))}^2}\\
\nonumber&\enspace& \quad +
C_M \deltat{n} \left( \| \deltat{}^2 \triangleL(\partial_{tt} u) \|_{L^{\infty}(I_n,H^1({\mathcal{D}}))}^2
+  \| \deltat{}^2 u  \|_{W^{2,\infty}(I_n,H^1(\mathcal{D}))}^2 + \| h u \|_{L^{\infty}(I_n,H^2({\mathcal{D}}))}^2 \right).
\end{eqnarray}
\end{lemma}

\begin{proof}
We proceed analogously to the proof of Lemma \ref{L2-estimate-step-n}.
Starting from the energy error identity (\ref{error-identity-enrgy}) we obtain the following estimates for the various terms. Using \eqref{equation-E-L} we get
\begin{eqnarray}
\label{est-energy-proof-step-1}\frac{| \Re \left(  \langle L( E_h^{n} + E_h^{n-1} ), P_{L^2}(\xi_n^{(1)}) \rangle \right) |}{ \energy{ E_h^{n} + E_h^{n-1} } } \overset{(\ref{H1-stability-L2-projection})}{\le} C_{L^2} \energy{ \xi_n^{(1)} }
\le C_{L^2} C_1 \deltat{n}^{1/2} \| h \partial_t u \|_{L^2(I_n,H^2({\mathcal{D}}))}.
\end{eqnarray}
Second, using $\| \PR(u^{n}) \|_{W^{1,\infty}({\mathcal{D}})} \le 
C_{W^{1,\infty}} (\mbox{diam}(\mathcal{D})+1) \| \nabla u^n \|_{L^{\infty}(\mathcal{D})} \le M$ (cf. (A7)) we get
\begin{eqnarray}
\nonumber\lefteqn{| \Re \left( \langle L( E_h^{n} + E_h^{n-1} ), P_{L^2}(\xi_n^{(2)}) \rangle \right) | }\\
\nonumber&\overset{(\ref{H1-stability-L2-projection})}{\le}& C_{L^2} \deltat{n}
 \energy{ \hat{f}( {\frac{\PR(u^{n}) +\PR(u^{n-1})}{2}} ) - \hat{f}({\frac{U^{n}+U^{n-1}}{2}}) }
\energy{ E_h^{n} + E_h^{n-1} }\\
\label{est-energy-proof-step-2}
\nonumber&\overset{(\ref{f_M_cond_3})}{\le}& C_{L^2} \left( c_{M} \beta \deltat{n}  (\energy{ E_h^{n} }^2 + \energy{ E_h^{n-1} }^2) + \deltat{n} \sum_{k=n-1}^n \energy{  \kappa (\PR(u^{k}) -  U^{k}) } \energy{ E_h^{n} + E_h^{n-1} } \right)\\
\nonumber&\overset{(\ref{constants-in-norm-equivalence})}{\le}&
 C_{L^2} \left( c_{M} \beta \deltat{n} (\energy{ E_h^{n} }^2 + \energy{ E_h^{n-1} }^2) \right.\\
&\enspace& \hspace{40pt} \left. + 2 C_{\mbox{\rm \tiny Sobolev}} \frac{\sqrt{C_E}}{\sqrt{c_E}}  \deltat{n} \| \kappa \|_{W^{1,3}({\mathcal{D}})}  (\energy{ E_h^{n} }^2 + \energy{ E_h^{n-1} }^2) \right).
\end{eqnarray}
In the last step we also used the following inequality (based on Sobolev embeddings) which holds for any $v \in H^1({\mathcal{D}})$
\begin{align*}
 \| \nabla(\kappa v) \|_{L^2({\mathcal{D}})} &\le \| (\nabla \kappa) v) \|_{L^2({\mathcal{D}})} + \| \kappa (\nabla v) \|_{L^2({\mathcal{D}})}
\le \| \nabla \kappa \|_{L^3({\mathcal{D}})} \| v \|_{L^6({\mathcal{D}})} + \| \kappa \|_{L^{\infty}({\mathcal{D}})} \| v \|_{H^1({\mathcal{D}})} \\
&\le \left( C_{\mbox{\rm \tiny Sobolev}} \| \kappa \|_{W^{1,3}({\mathcal{D}})} + \| \kappa \|_{L^{\infty}({\mathcal{D}})} \right) \| v \|_{H^1({\mathcal{D}})}.
\end{align*}
For the $\xi_n^{(3)}$ we use again that
$\| u \|_{L^{\infty}(I_n \times \mathcal{D})} \le M$ and
$\| P_h(u^n) \|_{L^{\infty}(\mathcal{D})} \le M$ in combination with $f_{M}(z)=|z|^2 z$ for $|z|\le M$. This yields
\begin{eqnarray}
\label{pre-est-energy-proof-step-3}\lefteqn{\frac{| \Re \left(  \langle L( E_h^{n} + E_h^{n-1} ) , P_{L^2}(\xi_n^{(3)}) \rangle \right) |}{ \energy{ E_h^{n} + E_h^{n-1} }} \le
C_{L^2} \frac{\sqrt{C_E}}{\sqrt{c_E}} \| \nabla \xi_n^{(3)} \|_{L^2(\mathcal{D})}}\\
\nonumber&\le& 
\underset{ \mbox{I}_{\xi_n^{(3)}} }{\underbrace{ C_{L^2} \frac{\sqrt{C_E}}{\sqrt{c_E}} \| \nabla \left( \int_{I_n} \kappa (u(\cdot,t) - ( \frac{\PR(u^{n}) +\PR(u^{n-1})}{2} ) \hspace{2pt} dt \right) \|_{L^2(\mathcal{D})} }}\\
\nonumber&\enspace& + \underset{ \mbox{II}_{\xi_n^{(3)}} }{\underbrace{\beta C_{L^2} \frac{\sqrt{C_E}}{\sqrt{c_E}} \| \nabla \left( \int_{I_n} f(u(\cdot,t)) - f( \frac{\PR(u^{n}) +\PR(u^{n-1})}{2} ) \hspace{2pt} dt \right) \|_{L^2(\mathcal{D})}}}.
\end{eqnarray}
The regularity of $\kappa$ (and the fact that we can hide $\| \kappa \|_{W^{1,3}({\mathcal{D}})}$ and $\| \kappa \|_{L^{\infty}({\mathcal{D}})}$ in $C_M$) allows us to estimate the first term by $\mbox{I}_{\xi_n^{(3)}} \le C_M \deltat{n} (  \| \deltat{}^2 \partial_{tt} u \|_{L^{\infty}(I_n,H^1({\mathcal{D}}))} + \| h u \|_{L^{\infty} (I_n,H^2({\mathcal{D}}))})$.
For the second term we get
\begin{eqnarray*}
\lefteqn{ \mbox{II}_{\xi_n^{(3)}} \le \underset{ \mbox{i}_{\xi_n^{(3)}} }{\underbrace{\beta C_{L^2} \frac{\sqrt{C_E}}{\sqrt{c_E}} \| \nabla \left( \int_{I_n} f(u(\cdot,t)) - f( \frac{u^{n} + u^{n-1}}{2} ) \hspace{2pt} dt \right) \|_{L^2(\mathcal{D})}}} }\\
&\enspace& \quad + \underset{ \mbox{ii}_{\xi_n^{(3)}} }{\underbrace{\deltat{n} \beta C_{L^2} \frac{\sqrt{C_E}}{\sqrt{c_E}} \| \nabla \left(  f( \frac{u^{n} + u^{n-1}}{2} ) - f( \frac{\PR(u^{n}) +\PR(u^{n-1})}{2} ) \right) \|_{L^2(\mathcal{D})}}},
\end{eqnarray*}
where we can use \eqref{estimate-xi-3-II-step-1-grad-NEW} to obtain
\begin{eqnarray*}
\mbox{i}_{\xi_n^{(3)}} &\le& \deltat{n}^3 C_M \| u \|_{W^{1,\infty}(I_n,H^1(\mathcal{D}))} + \deltat{n}^3 C_M \| u \|_{W^{2,\infty}(I_n,H^1(\mathcal{D}))}
\end{eqnarray*}
and where we can use $\| u^{n} \|_{W^{1,\infty}(\mathcal{D})}$, $\| \PR(u^{n}) \|_{W^{1,\infty}(\mathcal{D})} \le M$ to get
\begin{eqnarray*}
\mbox{ii}_{\xi_n^{(3)}}
&\le& C_M \deltat{n} \| h u \|_{L^{\infty}(I_n,H^2({\mathcal{D}}))}.
\end{eqnarray*}
Combining the estimates for $\mbox{I}_{\xi_n^{(3)}}$, $\mbox{II}_{\xi_n^{(3)}}$, $\mbox{i}_{\xi_n^{(3)}}$ and $\mbox{ii}_{\xi_n^{(3)}}$ with \eqref{pre-est-energy-proof-step-3} yields
\begin{eqnarray}
\label{est-energy-proof-step-3}\frac{| \Re \left(  \langle L( E_h^{n} + E_h^{n-1} ) , P_{L^2}(\xi_n^{(3)}) \rangle \right) |}{ \energy{ E_h^{n} + E_h^{n-1} }} \le C_M \deltat{n} \left( \| \deltat{}^2 u \|_{W^{2,\infty}(I_n,H^1(\mathcal{D}))}
+ \| h u \|_{L^{\infty}(I_n,H^2({\mathcal{D}}))} \right).
\end{eqnarray}
For the last term in the error identity \eqref{error-identity-enrgy} we get
\begin{eqnarray}
\label{est-energy-proof-step-4}
\lefteqn{
\frac{| \Re \left( \ci \langle L( E_h^{n} + E_h^{n-1} ), \int_{I_n}  \frac{P_{L^2} (\triangleL u^{n})+ P_{L^2} (\triangleL u^{n-1})}{2} - P_{L^2}( \triangleL u(\cdot,t)) \hspace{2pt} dt
\rangle \right) |}{ \energy{ E_h^{n} + E_h^{n-1} } }}\\
\nonumber&\le&
\energy{ P_{L^2} \left( \int_{I_n} \frac{\triangleL u^{n} + \triangleL u^{n-1}}{2} - \triangleL u(\cdot,t) \hspace{2pt} dt \right) }
\le \deltat{n}^{3} C_{L^2} \| \triangleL( \partial_{tt} u) \|_{L^{\infty}(I_n,E({\mathcal{D}}))}.
\end{eqnarray}
Combining estimates \eqref{est-energy-proof-step-1}-\eqref{est-energy-proof-step-4} and plugging them into the error identity (\ref{error-identity-enrgy}) finishes the proof.
\end{proof}

\begin{lemma}[Full $L^2$-error estimate for $E_h^{n}$]
\label{L2-estimate-full-E-h-n}
We use the notation and the assumptions of Lemma \ref{L2-estimate-step-n}.
Then it holds
\begin{eqnarray*}
\lefteqn{\| E_h^{n} \|_{L^2({\mathcal{D}})}^2
\le e^{4 C_M t_n} \| E_h^{0} \|_{L^2({\mathcal{D}})}^2
+ C_M e^{4 C_M t_n} \| h^2 \partial_t u \|_{L^2(0,t_n,H^2({\mathcal{D}}))}^2} \\
&\enspace& \quad +
C_M e^{4 C_M t_n} \sum_{k=1}^n \deltat{k} \left( \| \deltat{}^2 \triangleL(\partial_{tt} u) \|_{L^{\infty}(I_k,L^2({\mathcal{D}}))}^2
+  \| \deltat{}^2 u  \|_{W^{2,\infty}(I_k,L^2(\mathcal{D}))}^2 + \| h^2 u \|_{L^{\infty}(I_k,H^2({\mathcal{D}}))}^2 \right).
\end{eqnarray*}
\end{lemma}

\begin{proof}
First we note that if $a_n,b_n,\alpha_n$ is a sequence of positive real numbers that is related via $a_{n+1} \le (1 + \alpha_n) a_n + b_n$ then it holds
\begin{align}
\label{std-est-exp}
a_{n+1} &\le
e^{\sum_{i=0}^n \alpha_i} \left( a_0 + \sum_{i=0}^n b_i \right).
\end{align}
Next we use equation (\ref{prelim-error-estimate}) to obtain
\begin{eqnarray}
\label{est-step-new-1}\| E_h^{n} \|_{L^2({\mathcal{D}})}^2 &\le& (1 + \alpha_n ) \| E_h^{n-1} \|_{L^2({\mathcal{D}})}^2 + b_n,
\end{eqnarray}
where $\alpha_{n}:= \frac{ 2 C_M \deltat{n} }{(1- C_M \deltat{n} ) }$
and
\begin{align*}
b_n &:=  C_M \| h^2 \partial_t u \|_{L^2(I_n,H^2({\mathcal{D}}))}^2\\
&\enspace \quad +
C_M \deltat{n} \left( \| \deltat{}^2 \triangleL(\partial_{tt} u) \|_{L^{\infty}(I_n,L^2({\mathcal{D}}))}^2
+  \| \deltat{}^2 u  \|_{W^{2,\infty}(I_n,L^2(\mathcal{D}))}^2 + \| h^2 u \|_{L^{\infty}(I_n,H^2({\mathcal{D}}))}^2 \right).
\end{align*}
Combining (\ref{est-step-new-1}) with \eqref{std-est-exp} and $C_M \deltat{n} \le 1/2$ yields
\begin{eqnarray*}
\lefteqn{e^{- \sum_{k=1}^n \alpha_k} \| E_h^{n} \|_{L^2({\mathcal{D}})}^2
\le  \| E_h^{0} \|_{L^2({\mathcal{D}})}^2
+ C_M \sum_{k=1}^n \| h^2 \partial_t u \|_{L^2(I_k,H^2({\mathcal{D}}))}^2} \\
&\enspace& \quad +
C_M \sum_{k=1}^n \deltat{k} \left( \| \deltat{}^2 \triangleL(\partial_{tt} u) \|_{L^{\infty}(I_k,L^2({\mathcal{D}}))}^2
+  \| \deltat{}^2 u  \|_{W^{2,\infty}(I_k,L^2(\mathcal{D}))}^2 + \| h^2 u \|_{L^{\infty}(I_k,H^2({\mathcal{D}}))}^2 \right).
\end{eqnarray*}
The
inequality
$\sum_{k=1}^n \alpha_k \le 4 C_M t_n$ finishes the proof.
\end{proof}

\begin{lemma}[Full energy-error estimate for $E_h^{n}$]
\label{energy-estimate-full-E-h-n}
We use the notation and the assumptions of Lemma \ref{energy-estimate-step-n}). It holds
\begin{eqnarray*}
\lefteqn{ \energy{ E_h^{n}}^2
\le e^{4 C_M t_n} \energy{ E_h^{0} }^2
+ C_M e^{4 C_M t_n} \| h \partial_t u \|_{L^2(0,t_n,H^2({\mathcal{D}}))}^2} \\
&\enspace& \quad +
C_M e^{4 C_M t_n} \sum_{k=1}^n \deltat{k} \left(  \| \deltat{}^2 \triangleL(\partial_{tt} u) \|_{L^{\infty}(I_k,H^1({\mathcal{D}}))}^2
+ \| \deltat{}^2 u  \|_{W^{2,\infty}(I_k,H^1(\mathcal{D}))}^2 + \| h u \|_{L^{\infty}(I_k,H^2({\mathcal{D}}))}^2 \right).
\end{eqnarray*}
\end{lemma}

\begin{proof}
The proof is analogous to the proof of Lemma \ref{L2-estimate-full-E-h-n} by combining equation (\ref{std-est-exp}) with Lemma \ref{energy-estimate-step-n}.
\end{proof}

Following the ideas of \cite{KaM98}, we want to show that the solution $u_h^n$ of the original discrete problem (\ref{cnd-problem}) is identical to the solution $U^n_h$ of the auxiliary problem (\ref{aux-problem-bwe}) implying that the estimates in Lemma \ref{L2-estimate-full-E-h-n} and \ref{energy-estimate-full-E-h-n} hold equally for $u_h^n$. For that purpose, we want to show that if $\deltat{n}$ is sufficiently small
it holds $\| U^n_h \|_{L^{\infty}({\mathcal{D}})} \le M$ for all $n \ge 0$. Then, by the properties of $f_M$, we obtain equality of $u_h^n$ and $U_h^n$. To show the desired boundedness we can use again Lemma \ref{lemma-thomee}, which guarantees
$\| v_h \|_{L^{\infty}({\mathcal{D}})} \le  C_{\infty} \ell_h \| \nabla v_h \|_{L^2({\mathcal{D}})}$
for all $v_h\in S_h$.

\begin{conclusion}
\label{existence-of-M}
Let assumptions (A1)-(A7) be fulfilled and let $h$ and $\deltat{n}$ be such that $\ell_h (\hmax+\deltat{n}^2) \rightarrow 0$ for $h,\deltat{n}\rightarrow 0$. Then, for all small enough $h$ and $\deltat{n}$, the corresponding solution $U^n_h$ (i.e. the solution for $f_M$ as specified in Lemma \ref{makridakis-lemma}) fulfills 
$$\| U^n_h \|_{L^{\infty}({\mathcal{D}})} \le M.$$
\end{conclusion}

\begin{proof}
We have $U^n_h = E_h^n + \PR(u^n)$. Using (\ref{constants-in-norm-equivalence}) and Lemma \ref{lemma-thomee} we get
\begin{align*}
\| U^n_h \|_{L^{\infty}({\mathcal{D}})} &\le \| \PR(u^n) \|_{L^{\infty}({\mathcal{D}})} + C_{\infty} \sqrt{C_E} \ell_h \energy{E_h^n}.
\end{align*}
The term $\PR(u^n)$ is uniformly bounded by 
(A7)
and the Poincar\'e-Friedrichs inequality with $\| P_h(u^n) \|_{L^{\infty}(\mathcal{D})} \le C_{W^{1,\infty}} \mbox{diam}(\mathcal{D}) \| \nabla u^n \|_{L^{\infty}(\mathcal{D})}$. Let us hence consider the second term. Fixing the model problem (and assuming (A1)-(A7)), the only variables are $h$ and $\deltat{n}$.
With this, we can write Lemma \ref{energy-estimate-full-E-h-n} as: there exists a constant $C(M)$, which is independent of $h$ and $\deltat{n}$ such that
\begin{align*}
 \energy{E_h^n} \le C(M) (\hmax+\underset{1\le n\le N}{\max}\deltat{n}^2).
\end{align*}
Consequently, for each given 
$\epsilon>0$, we can pick $h(M)$ and $\deltat{n}(M)$ small enough so that
\begin{align*}
 \ell_h(M) \energy{E_h^n} \le C(M) \ell_h(M) (\hmax(M)+\underset{1\le n\le N}{\max}\deltat{n}^2(M) \le \epsilon.
\end{align*}
If we choose 
$h$ and $\deltat{n}$ small enough so that $C_{\infty} \sqrt{C_E} \ell_h(M) \energy{E_h^n} \le \| u \|_{W^{1,\infty}(I_n, W^{1,\infty}(\mathcal{D}))}$, then we obtain $\| U^n_h \|_{L^{\infty}({\mathcal{D}})} \le M$ as desired.
\end{proof}

Observe that Conclusion \ref{existence-of-M} proves Proposition \ref{prop-exist-and-unique}.
We are now prepared to conclude the proof of Theorem \ref{final-a-priori-error-estimate}.

\begin{proof}[Proof of Theorem \ref{final-a-priori-error-estimate}]
We pick 
$h(M)$ and $\deltat{n}(M)$ small enough so that the bound in
Conclusion \ref{existence-of-M} 
holds true.
Since $\| U^n_h \|_{L^{\infty}({\mathcal{D}})} \le M$ we obtain from the properties of $f_M$ (see Lemma \ref{makridakis-lemma}) that $U^n_h$ must be identical to the solution $u_h^n$ of (\ref{cnd-problem}) for every time step $n\ge1$. Hence, we obtain the splitting
\begin{align*}
u_h^n - u^n  = E_h^n +  (\PR(u^n) -  u^n),
\end{align*}
where $E_h^n$ can be estimated by Lemma \ref{L2-estimate-full-E-h-n}, respectively Lemma \ref{energy-estimate-full-E-h-n} and where $(\PR(u^n) -  u^n)$ can be estimated in the usual matter. A Lagrange-interpolation error estimate for the initial value $u_0\in H^2({\mathcal{D}})$ concludes the proof.
\end{proof}

\section{Numerical experiments}
\label{section-num-sec}

In this section we investigate the performance of the one-stage Gauss-Legendre implicit Runge-Kutta scheme stated in Definition \ref{crank-nic-gpe}
and compare it with the approximations obtained with the Backward-Euler method \eqref{bwe-problem} to stress the importance of the discrete mass conservation. We consider the computational domain $\mathcal{D}:=[-6,6]^2$ and the time interval $[0,T_{\mbox{\tiny max}}]:=[0,100]$. We seek a solution $u \in C^0([0,T),H^1_0({\mathcal{D}}))$ to the time-dependent Gross-Pitaevskii equation
\begin{align}
\label{modprobnum}\ci \partial_t u &= -\frac{1}{2} \triangle u + V \hspace{2pt} u  - \Omega \mathcal{L}_z u + \beta |u|^2 u \qquad \mbox{in} \enspace {\mathcal{D}},
\end{align}
where we recall $\mathcal{L}_z = - \ci \left( x \partial_y - y \partial_x \right)$. We use the following configuration. We chose $\beta=100$, $\Omega=0.8$ and the harmonic potential
$$V(\mathbf{x}):=\frac{\gamma_x^2 x^2+ \gamma_y^2 y^2}{2}$$
with trapping frequencies $\gamma_x=0.9$ and $\gamma_x=1.1$. The initial value $u_0=u(\cdot,0)$ is chosen as the $L^2$-normalized ground state eigenvector of the Gross-Pitaevskii operator $G_0(v):=-\frac{1}{2} \triangle v + V_0 \hspace{2pt} v  - 0.8 \mathcal{L}_z v + 100 |v|^2 v$ with $V_0(\mathbf{x})=\frac{1}{2}(x^2+y^2)$ and corresponding ground state energy $E_0=3.1938$ (cf. \cite{BWM05}). We computed this ground state using the Discrete Normalized Gradient Flow method proposed in \cite{BaD04}. Starting from this setting, we wish to simulate the dynamics of $u_0$ in the anisotropic trap $V$, i.e. we solve \eqref{modprobnum} numerically. The problem is picked in such a way that vortices, i.e. density singularities, are formed in the condensate (see Figure \ref{comp-cn-be-at-T-1} and \ref{cn-series}). We define the energy by
$$\mathcal{E}(v):= (v,v)_{E({\mathcal{D}})} + \int_{\mathcal{D}} ( \kappa |v|^2 + \frac{\beta}{2} |v|^4),$$
which is a conservative property of equation \eqref{modprobnum}.

\begin{figure}[h!]
\centering
\includegraphics[scale=0.35]{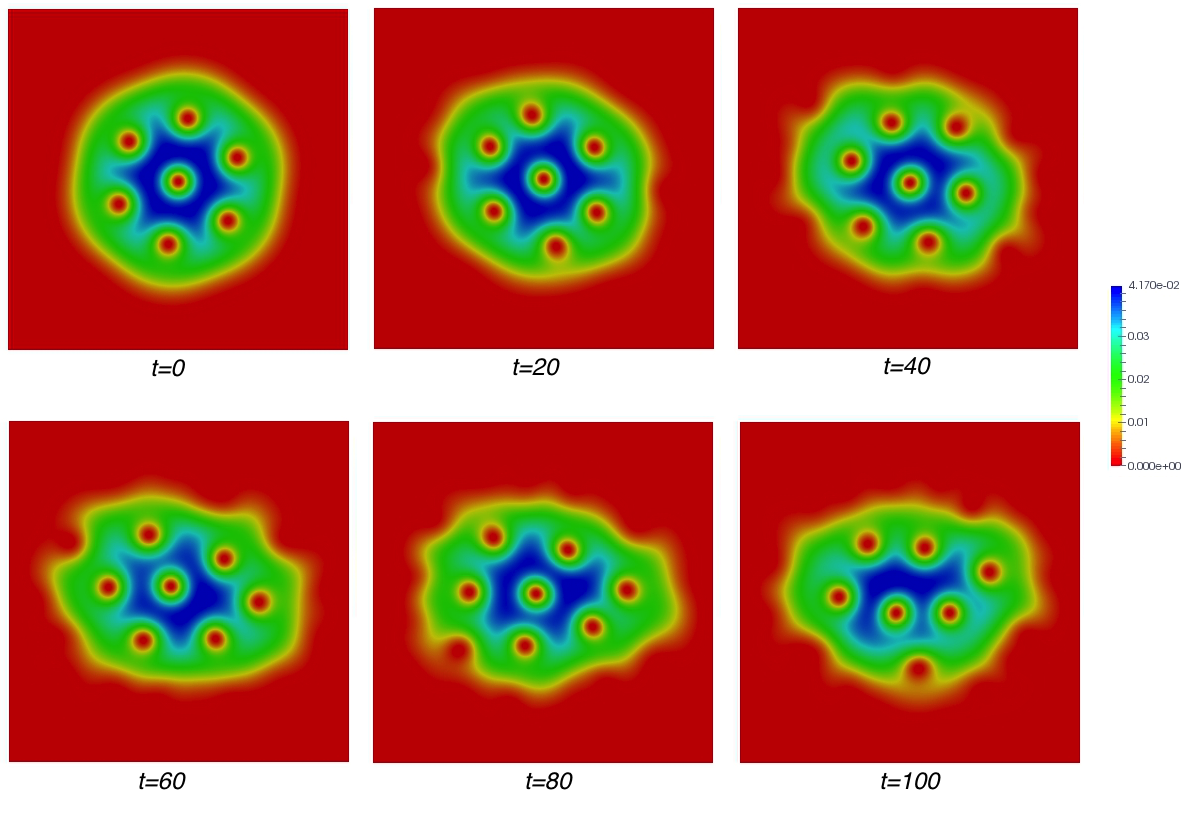}
\caption{\it The figures show the approximations for the particle density obtained with the Crank-Nicolson scheme for large time steps $\deltat{}=0.1$ at times $t=0$, $t=20$, $t=40$, $t=60$, $t=80$ and $t=100$. We observe a reduction of the number of vortices. The mass is fully preserved and the energy up to a relative error of $0.026\%$.}
\label{cn-series}
\end{figure}

\begin{table}[h!]
\caption{\it The table shows the loss in energy and mass of the Backward-Euler scheme after $100$ time steps for different time step sizes $\deltat{}$, compared wth the corresponding quantities obtained with the Crank-Nicolson-type IRK scheme. Recall that the IRK always conserves the mass.}
\label{table-comparison-energy}
\begin{center}
\begin{tabular}{|c|c||c|c||c|c|}
\hline $\deltat{}$      & $T$
& $\|u_{h,BE}^N\|_{L^2(\mathcal{D})}$
& $\|u_{h,IRK}^N\|_{L^2(\mathcal{D})}$
& $\mathcal{E}(u_{h,BE}^N)$
& $\mathcal{E}(u_{h,IRK}^N)$ \\
\hline
\hline $1$    & 100   & $2.6 \cdot 10^{-17}$ & 1.0 & $6.9 \cdot 10^{-34}$ & 3.36455 \\ %
\hline $10^{-1}$ & 10     & $0.13275$ & 1.0 & 0.02309 & 3.19203 \\
\hline $10^{-2}$ & 1       & $0.91676$ & 1.0 & 2.52235 & 3.19383 \\
\hline $10^{-3}$ & 0.1    & $0.99905$ & 1.0 & 3.18549 & 3.19383 \\
\hline
\end{tabular}\end{center}
\end{table}

\begin{figure}[h!]
\centering
\includegraphics[scale=0.35]{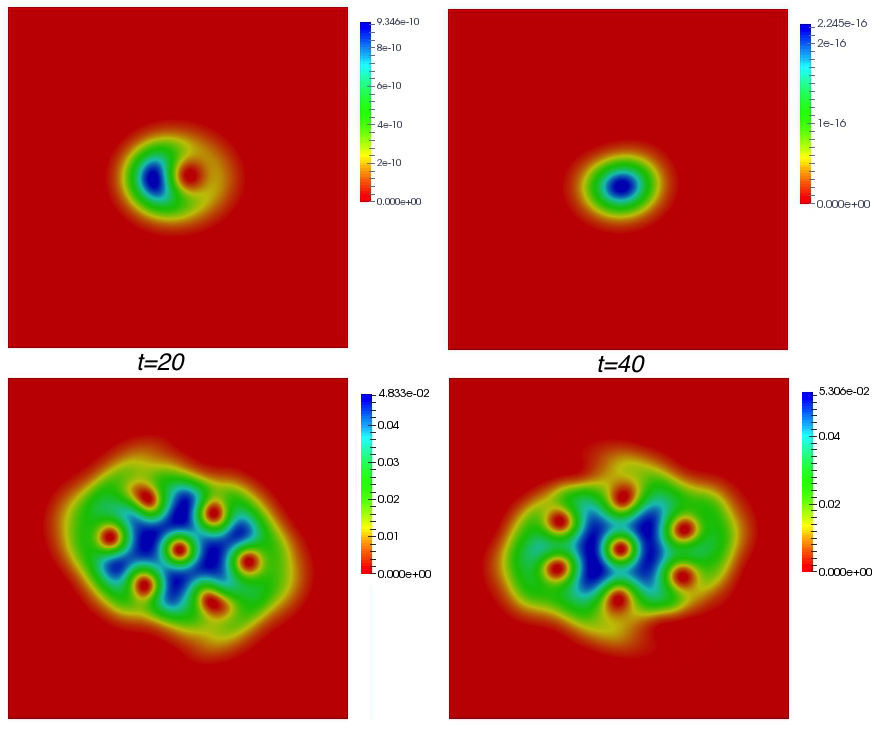}
\caption{\it The figures shows the approximations for the particle density obtained with the Backward-Euler (upper row) and the IRK scheme (lower row) for the time step size $\deltat{}=1$ at times $t=20$ and $t=40$.}
\label{be-series}
\end{figure}

\begin{figure}[h!]
\centering
\includegraphics[scale=0.38]{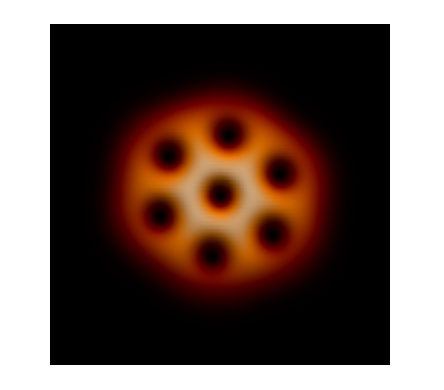}
\hspace{0pt}
\includegraphics[scale=0.38]{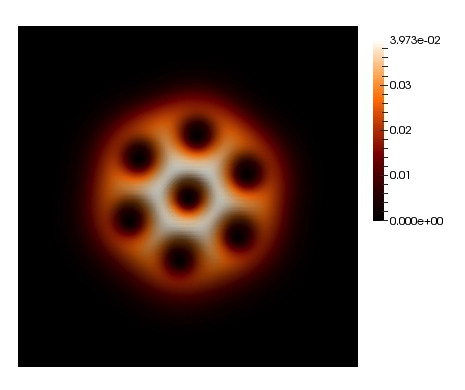}
\caption{\it The figure depicts numerical approximations for the particle density $\rho(\cdot,T):=|u(\cdot,T)|^2$ at $T=1$. The left approximation is a Backward-Euler approximation computed with time step size $\deltat{}=10^{-2}$, whereas the IRK approximation on the right is computed in one time step, i.e. with $\deltat{}=1$.}
\label{comp-cn-be-at-T-1}
\end{figure}

We demonstrate the efficiency of the Crank-Nicolson-type IRK scheme (as stated in Definition \ref{crank-nic-gpe}) by showing that large time steps are allowed, thanks to the mass conservation property. The Backward-Euler approach on the other hand (despite being unconditionally stable) does not allow large time steps, since this results in a severe loss of mass which lets the corresponding approximations vanish quickly. In all our computations we use a uniformly refined triangular mesh $\T_h$ with $66.049$ nodes. That means that the discrete space $S_h$ contains $132.098$ degrees of freedom (minus the ones from the boundary condition). We use uniform time steps and denote $\deltat{}:=\deltat{n}$ for simplicity.

We note that the computational complexity of the IRK scheme \eqref{cnd-problem} and the Backward-Euler scheme \eqref{bwe-problem} is roughly the same in our implementation. Since both schemes are implicit, they require an iterative Newton method in each time step where we observe a comparable number of iterations to reach a given tolerance. In the following, we shall denote Backward-Euler approximations by $u_{h,BE}^n$ and IRK approximations by $u_{h,IRK}^n$.

Due to the structure of the problem, we could use exact integration when assembling the system matrices and load vector for our problem. Furthermore, all linear systems were solved with an UMFPACK direct solver. The only reason why we were not computationally exact (up to machine precision), was that we prescribed a residual tolerance of order $\mathcal{O}(10^{-8})$ for the Newton-algorithm to abort. This inexactness did not have an observable effect on the conservation of mass for the IRK in any of the computations. Concerning the energy, a small deviation from the exact energy was observable over time for the IRK, however in a negligible range. For instance, for large steps $\deltat{}=1$, the energy was still preserved up to an error of $5.3\%$ at $T=100$. For slightly smaller time steps with $\deltat{}=0.1$ the conservation of energy already improved to a relative error of below $0.03\%$ which is insignificant considering that the reference energy (at $t=0$) is typically already polluted by discretization errors.
Using a time step size $\deltat{}=0.1$ we simulated the dynamics of the particle density on the time interval $[0,100]$. The corresponding results are depicted in Figure \ref{cn-series}. We observe that the condensate with initially seven vortices collapses to a condensate with six vortices at $T=100$.

This is in strong contrast to the Backward-Euler scheme that, though unconditionally stable, suffers from a major loss of energy and mass. This is clearly shown in Table \ref{table-comparison-energy}. For time step sizes of order $\deltat{}=1$, basically all energy and mass is lost after $100$ time steps. In our example the situation gradually improves with decreasing time steps sizes, however, to obtain an acceptable loss of mass and energy after 100 time steps, the Backward-Euler method requires time steps sizes of at least $\deltat{}=10^{-3}$. The significance of the preservation properties is further emphasized in Figure \ref{be-series}. Here we compare Backward-Euler and IRK approximations for large time steps $\deltat{}=1$. We can see that even though $|u_{h,IRK}^n|^2$ is not particularly accurate, it still preserves the structure of the condensate, whereas $|u_{h,BE}^n|^2$ quickly collapses into a vanishing mass that fully contradicts the correct physical behavior.

Finally, in Figure \ref{comp-cn-be-at-T-1} we compare the IRK approximation after one single step of order $\deltat{}=1$ with the Backward-Euler approximation at the same time ($T=1$) but using $100$ time steps with size $\deltat{}=10^{-2}$ each. Even though the costs for the Backward-Euler scheme are a $100$ times higher as for the IRK approach to obtain a comparable result, the approximation $|u_{h,BE}^n|^2$ has clearly not yet the quality of $|u_{h,IRK}^n|^2$.

In summary we conclude that even though the Backward-Euler scheme seems to be unconditionally stable, the loss of mass and energy has a tremendous impact on the quality of the obtained approximations if the time-step size is not chosen very small. The IRK scheme of Crank-Nicolson-type on the other hand does not appear to have such restrictions.

\medskip
$\\$
{\bf Acknowledgements.}
The authors thank the anonymous referees for their very valuable comments on the initial version of this manuscript, which helped to improve the contents of this paper significantly. Furthermore, the authors thank J\"urgen Ro{\ss}mann for the helpful and enlightening discussions on H\"older-estimates for the Green's function of $L$ as they influence the validity of assumption (A7).

\def\cprime{$'$}

\end{document}